\documentclass{amsart}
\usepackage{amsmath}                
\usepackage{graphicx}
\usepackage{mathrsfs}
\usepackage{amssymb}
\usepackage{esint}
\usepackage{epstopdf}
\usepackage{tikz-cd}
\usepackage{caption}
\usepackage{subcaption}
\usepackage{changes}

\usepackage{MnSymbol}
\usepackage{amsthm}
\theoremstyle{plain}

\newtheorem{Lemma}{Lemma}[section]
\newtheorem*{Lemma*}{Lemma}
\newtheorem{Theorem}{Theorem}
\newtheorem*{Theorem*}{Theorem}
\newtheorem{Proposition}[Lemma]{Proposition}
\newtheorem*{Proposition*}{Proposition}

\newtheorem{Remark}[Lemma]{Remark}
\newtheorem*{Remark*}{Remark}

\newtheorem{Hypothesis}[Lemma]{Hypothesis}
\newtheorem*{Hypothesis*}{Hypothesis}

\newcommand{\R}{\mathbb{R}}
\newcommand{\C}{\mathbb{C}}
\newcommand{\N}{\mathbb{N}}
\newcommand{\Z}{\mathbb{Z}}
\newcommand{\bbP}{\mathbb{P}}
\newcommand{\bbH}{\mathbb{H}}

\newcommand{\mmH}{\mathcal{H}}

\newcommand{\eps}{\varepsilon}

\def\Re{\mathop{\mathrm{Re}}}

\def\coker{\mathop{\mathrm{Coker}\,}}

\newcommand{\rmO}{\mathrm{O}}
\newcommand{\rmo}{\mathrm{o}}

\newcommand{\rme}{\mathrm{e}}
\newcommand{\rmi}{\mathrm{i}}
\renewcommand{\ker}{\mathrm{Ker}\,}
\newcommand{\Id}{\mathrm{\,Id}\,}

\newcommand{\supp}{\mathrm{supp}}

\usepackage[labelformat=simple,labelfont={}]{subcaption}

\begin{document}

\author{Gabriela Jaramillo and Shankar C. Venkataramani}

\address{\noindent Department of Mathematics, University of Arizona,
  Tucson, AZ  85721
\newline
e-mail:  \rm \texttt{gjaramillo@math.arizona.edu}
}

\address{\noindent Department of Mathematics, University of Arizona,
  Tucson, AZ  85721
\newline
e-mail:  \rm \texttt{shankar@math.arizona.edu}
}

\title[Target patterns with nonlocal coupling]{Target Patterns in a 2-d Array of Oscillators with Nonlocal Coupling}

\begin{abstract}
\noindent
We analyze the effect of adding a weak, localized, inhomogeneity to a two dimensional array of oscillators with nonlocal coupling. We propose and also justify a model for the phase dynamics in this system. Our model is a generalization of a viscous eikonal equation that is known to describe the phase modulation of traveling waves in reaction-diffusion systems. We show the existence of a branch of  target pattern solutions that bifurcates from the spatially homogeneous state when $\eps$, the strength of the inhomogeneity, is nonzero and we also show that these target patterns have an asymptotic wavenumber that is small beyond all orders in $\eps$. 

The strategy of our proof is to pose a good ansatz for an approximate form of the solution and use the implicit function theorem to prove the existence of a solution in its vicinity. The analysis presents two challenges. First, the linearization about the homogeneous state is a convolution operator of diffusive type and hence not invertible on the usual Sobolev spaces. Second, a regular perturbation expansion in $\eps$ does not provide a good ansatz for applying the implicit function theorem since the nonlinearities play a major role in determining the relevant approximation, which also needs to be ``correct" to all orders in $\eps$. We overcome these two points by proving Fredholm properties for the linearization in appropriate Kondratiev spaces and using a refined ansatz for the approximate solution, which obtained using matched asymptotics.
\end{abstract}



\maketitle

\section{Introduction}\label{s:introduction}

Reaction-Diffusion equations describe the evolution of quantities $u(x,t)$ that are governed by  ``local" nonlinear dynamics, given by a reaction term $F(u)$, coupled with Fickian diffusion, 
\begin{equation}
\frac{\partial u}{\partial t}  = D \Delta u + F(u).
\label{e:reaction-diffusion}
\end{equation}
They are generic models for patterns forming systems and have applications to a wide range of phenomena from population biology \cite{fisher1937wave,kolmogorov1937study}, chemical reactions \cite{winfree1973scroll,kopell1981target,meron1992pattern,tyson1980target}, fluid  \cite{newell1969finite,segel1969distant} and granular \cite{eggers1999continuum,aranson2006patterns} flow patterns,  and in waves in neural  
\cite{fitzhugh1961impulses,nagumo1962active,NEURON_book} and in 
cardiac \cite{gray1995mechanisms} tissue. 

The onset of self-organized patterns in reaction-diffusion models typically corresponds to a bifurcation where a steady equilibrium for the local dynamics $\dot{u} = F(u)$ loses stability either through a pitchfork bifurcation, giving a {\em bistable medium}, or through a Hopf bifurcation, giving an {\em oscillatory medium}. 
In the latter case, a defining feature is the occurrence of temporally periodic, spatially homogeneous states, and an analysis of the symmetries of the system show that they can generically give rise to traveling waves, spiral waves, and target patterns \cite{golubitsky1988singularities,doelman2005dynamics}. A prototypical example of these behaviors is the Belousov-Zhabotinsky reaction where chemical oscillations are manifested as a change in the color of the solution. This system displays self-organized spiral waves, i.e. they form without any external forcing or perturbation,  as well as target patterns, which in contrast form when an impurity is present at the center of the pattern \cite{stich2006target, pagola1987wave,tyson1980target}. 

As already mentioned, these patterns are not specific to oscillating chemical reactions and in fact can be seen in any spatially extended oscillatory medium. In this more general setting, and for the particular case of target patterns, an impurity or {\em defect} constitutes a localized region where the system is oscillating at a slightly different frequency from the rest of the medium. Depending on the sign of the frequency shift,  and for systems of dimensions $d\leq 2$, these defects can act as {\em pacemakers} and generate waves that propagate away from the impurity. In the case when $d=2$, these waves are seen as concentric circular patterns that propagate away from the defect.

Showing the existence of these target pattern solutions in reaction-diffusion systems, and related amplitude equations, has been the subject of extensive research, see \cite{kopell1981target, tyson1980, hagan1981, nagashima1991, moriyama1995, golubitsky2000, stich2006target, kollar2007coherent, jaramillo2015inhomogeneities} for some examples as well as the reference in \cite{cross1993}. Mathematically one can describe these patterns as modulated wave trains which correspond to solutions to ~\eqref{e:reaction-diffusion} of the form $u({\bf x},t) = u_*(\phi({\bf x},t); \nabla \phi({\bf x},t))$, where $u_*({\bf \xi};{\bf k})$ is a $2\pi$ periodic function of ${\bf \xi}$ that depends on the  wavenumber, ${\bf k}$ \cite{kopell1981target,doelman2005dynamics,kollar2007coherent}. In other words, these patterns correspond to periodic traveling waves, whose phase varies slowly in time and space. Using multiple scale analysis one can show that the evolution of the phase, $\phi$, over long times is given by the {\em viscous eikonal equation},
\[ \phi_t= \Delta \phi - |\nabla \phi|^2.\]
This equation has been studied extensively in the physics literature, starting with the work of Kuramoto \cite{kuramoto1976}. More recently it was shown that it does indeed provide a valid approximation for the phase modulation of the patterns seen in oscillatory media \cite{doelman2005dynamics}.

Since oscillating chemical reaction can be thought of as a continuum of diffusively coupled oscillators, it is not surprising that an analogue of the above equation can also be derived as a description for the phase dynamics for an array of oscillators. Indeed, in Appendix \ref{sec:hierarchy} we formally show that the following integro-differential equation provides a phase approximation for a slow-time, O(1) in space, description of nonlocally coupled oscillators,

\begin{equation}\label{e:nonloc1}
 \phi_t = \mathcal{L} \ast \phi- |\mathcal{J} \ast \nabla \phi|^2 + \eps g(x,y) \quad (x,y) \in \R^2.
 \end{equation}
 This equation will be the focus of our paper. Here the operators $\mathcal{L}$ and $\mathcal{J}$ are spatial kernels that 
 depend on the underlying nonlocal coupling between the oscillators. In particular, the convolution kernel $\mathcal{L}$ models the nonlocal coupling of these phase oscillators and can be thought of as an analog of $\Delta$. Similarly, the term $|\mathcal{J}\ast \nabla \phi|^2 = \mathcal{J} \ast \nabla \phi \cdot \mathcal{J} \ast \nabla \phi$  represents nonlocal transport along diffused gradients and is a generalization of the quadratic nonlinearity of the viscous eikonal equation. Finally, the function $g(x,y)$ represents an inhomogeneity that perturbs the ``local" frequency of the oscillators. Notice that the viscous eikonal equation is recovered when  $\mathcal{L}$ and $\mathcal{J}$ are the Laplacian and the identity operator, respectively. This framework also incorporates other models for spatio-temporal pattern formation, including the Kuramoto-Sivashinsky equation which corresponds to $\mathcal{L} = -\Delta - \Delta^2, \,\mathcal{J} = \mathrm{Id}$.
  
 Our mathematical motivation for studying the above model comes from its nonlocal aspect and the resulting analytical challenges.  The approach we propose for studying the operator $\mathcal{L}$ is novel and could be adapted to study other problems which involve similar convolution operators. In particular, the type of linear operators that we will consider are a generalization of the kernels used in neural field models or  continuum coupled models of granular flow.  Just as in reaction-diffusion systems, these models exhibit  spatio-temporal periodic patterns, bumps, and traveling waves, (see  \cite{coombes2005waves, ermentrout1998neural, pinto2001spatially, kilpatrick2010effects} for the case of neural field models, and  \cite{umbanhowar1998periodic,venkataramani1998spatiotemporal,venkataramani2001pattern,aranson2006patterns} for the case of patterns in granular flows). In particular,  among the examples of traveling waves seen in experiments and replicated in the neural field models are spirals and target patterns \cite{huang2004spiral, kilpatrick2010spatially, folias2004breathing}, which are of interest to us.

The challenge however is that,  like our model~\eqref{e:nonloc1}, these systems are not amenable to methods from spatial dynamics, which are typically used to study these phenomena. So we look for a more functional analytic approach, e.g. using the implicit function theorem, for proving the existence of solutions. As with other integro-differential equations the difficulty  comes from the linearization, which is a noncompact convolution operator, and in general not invertible when considered as a map between Sobolev spaces. This is a significant analytical challenge, and one way to overcome this difficulty is to use specific convolution kernels that allow for these integro-differential equations to be converted into PDEs via the Fourier Transform \cite{laing2003pde}, or in the radially symmetric case use sums of modified Bessel functions as models for synaptic footprint to simplify the analysis \cite{folias2004breathing}. 

The approach we consider in this paper is broader as it allows us to consider a larger class of convolution kernels by showing that these operators are Fredholm in appropriate weighted spaces. This approach is similar to the one in \cite{jaramillo2016pacemakers}, where we treated the one dimensional case and showed existence of target patterns in a large one dimensional array of oscillators with nonlocal coupling. For the applications we have in mind, e.g. neural field models, we need to extend these results to two dimensional arrays. 

The two dimensional case is technically more interesting because a regular perturbation expansion in $\eps$ does not always provide the correct ansatz. Indeed, in the case with local coupling, the equation
\[ -\omega = \Delta \tilde{\phi} - |\nabla \tilde{\phi}|^2 + \eps g(x,y), \quad (x,y) \in \R^2\]
 which results from inserting the ansatz $\phi(x,y,t) = \tilde{\phi}(x,y) - \omega t$ into the perturbed viscous eikonal equation, is conjugate to a Schr\"odinger eigenvalue problem via the Hopf-Cole transform, $\tilde{\phi} = - \ln (\Psi)$:
 \[ \omega \Psi  = \Delta \Psi - \eps g(x,y) \Psi, \quad (x,y) \in \R^2.\]
 In two dimensions, it is well known that the Schr\"odinger eigenvalue problem has bound states  if $\eps \int g < 0$ \cite{simon1976bound}. Notice that the ground state eigenfunction $\Psi_0$ can be chosen to be everywhere positive, so that $-\ln(\Psi_0)- \omega t$ does define a phase function $\phi$ solving the viscous eikonal equation with inhomogeneity. At the same time, the eigenvalue corresponding to the ground state is small beyond all orders of $\eps$ (see \cite{simon1976bound} and Section~\ref{s:eikonal} below), and is therefore not accessible to a regular perturbation expansion. This is the other analytical challenge that we have to overcome, and our approach is to develop a {\em superasymptotic perturbation expansion} for $\phi$, i.e. an approximation whose error is  $\rmO(\exp(-|c/\eps|))$ and captures behaviors that are small beyond all orders in $\eps$, \cite{boyd1999asymptotics}. 
 
 To show the existence of traveling waves for equation \eqref{e:nonloc1} we make the following assumptions on the convolution kernels $\mathcal{L}$ and $J$. First, we assume the kernel $\mathcal{L}$ is a diffusive and exponentially localized kernel that commutes with rotations. Consequently, its Fourier symbol $L$  depends only on $\xi = | {\bf k}|^2$.  We also impose additional properties that we specify in the Hypotheses \ref{h:analyticity} and \ref{h:multiplicity}. A representative example to keep in mind throughout the paper is the convolution kernel that would result in the formal operator $\Delta(\Id - \Delta)^{-1}$. 
 
 We reiterate that the model~\eqref{e:nonloc1} is derived under the assumption that the phase $\phi(x,t)$ varies slowly in time, {\em with no assumptions on its spatial variation}. If we assume that the solutions also vary slowly in space, then hypothesis~\ref{h:multiplicity} implies that the nonlocal operator $\mathcal{L}$ can be (formally) replaced by $\Delta$, and~\eqref{e:nonloc1} reduces to the ``local" viscous eikonal equation. Indeed, this is the setting for a substantial body of work on weakly coupled nonlinear oscillators \cite{kuramoto1984book, schwemmer2012theory}. Our additional contribution is that we  rigorously show the existence of target solutions of~\eqref{e:nonloc1} that vary slowly (on a scale $\sim e^{1/\epsilon}$) in space and time, if the model satisfies:
 
  \begin{Hypothesis}\label{h:analyticity}
  The  multiplication operator $L$ is a function of $\xi := | {\bf k} |^2$. Its domain can be extended to a strip in the complex plane, $\Omega = \R \times (-\rmi \xi_0, \rmi \xi_0)$ for some sufficiently small and positive $\xi_0 \in \R$, and on this domain the operator is uniformly bounded and analytic. Moreover, there is a constant $\xi_m \in \R$ such that the operator $L(\xi)$ is invertible with uniform bounds for $|\Re \xi| > \xi_m$.
  \end{Hypothesis}
 
 Our main result, Theorem~\ref{t:main} requires $\ell =1$ in the following hypothesis. We state the hypothesis in more generality because some of the intermediate results also hold more generally with $\ell \geq 1$.
 \begin{Hypothesis}\label{h:multiplicity}
 The multiplication operator $L(\xi)$ has a zero, $\xi^*$, of multiplicity $\ell \geq 1$ which we assume is at the origin. Therefore, the symbol $L(\xi)$ admits the following Taylor expansion near the origin.
  \[ L(\xi ) = (- \xi)^\ell + \rmO(\xi^{\ell+1}), \quad \mbox{for} \quad \xi \sim 0.\]
 \end{Hypothesis}
 
\begin{Hypothesis} \label{h:nonlinearity}
The kernel $\mathcal{J}$ is radially symmetric, exponentially localized, twice continuously differentiable, and 
\[ \int_{\R^2} \mathcal{J}({\bf x}) \;d{\bf x}  =1 \]
\end{Hypothesis}

Our strategy to show the existence of traveling waves will be to first establish the Fredholm properties of the convolution operator $\mathcal{L}$ following the ideas described in \cite{jaramillo2016effect}. This will allow us to precondition our equation by an operator $\mathscr{M}$, resulting in an equation which has as its linear part the Laplace operator. We then proceed to show the existence of target patterns in the nonlocal problem. More precisely, we prove Theorem \ref{t:main} where we use the following {\bf notation}:
\begin{itemize}
\item Here $L^2_\sigma(\R^2)$ denotes the $L^2$ space with weight $(1 + |{\bf x}|^2)^{\sigma/2}$.
\item Similarly, $H^s_\sigma(\R^2)$ denotes the Hilbert space $H^s$ with weight $(1 + |{\bf x}|^2)^{\sigma/2}$.
\item Lastly, the symbol $M^{s,p}_\sigma(\R^2)$ describes the completion of $C_0(\R^2)$ functions under the norm
\[ \| u\|^p_{M^{s,p}_\sigma(\R^d)} = \sum_{|\alpha|\leq s} \left \| D^{\alpha}u({\bf x}) \cdot (1+|{\bf x}|^2)^{(\sigma + |\alpha|)/2} \right \|^p_{L^p(\R^d)}. \]
\end{itemize}
We describe these last spaces with more detail in Section \ref{s:weightedspaces}.

\begin{Theorem}\label{t:main}
Suppose that the kernels $\mathcal{L}$ and $\mathcal{J}$ satisfy Hypotheses~\ref{h:analyticity},~\ref{h:multiplicity} with $\ell =1$, and~\ref{h:nonlinearity}. Additionally, suppose $g$ is in the space $  L^2_\sigma(\R^2)$ with $\sigma>1$ and let $M= \frac{1}{2\pi} \int_{\R^2} g <0$. Then, there exists a number $\eps_0>0$ and a $C^1$ map
\begin{equation*}
	\begin{matrix}
	\Gamma:& [0,\eps_0) &\longrightarrow & \mathcal{D} \subset M^{2,1}_{\gamma-1}(\R^2)\\
	& \eps& \longmapsto & \psi\\
	\end{matrix}
\end{equation*}
where $\gamma>0$ and $ \mathcal{D} = \{\phi \in M^{1,2}_{\gamma-1}(\R^2) \mid  \nabla\phi \in H^1_\gamma(\R^2)  \}$, that allows us to construct an $\eps$-dependent family of target pattern solutions to \eqref{e:nonloc1}. Moreover, these solutions have the form
\[\Phi(r,\theta, t;\eps) = -\chi(\lambda(\eps) r) \ln( K_0(\lambda(\eps) r) ) + \psi(r, \theta ;\eps)- \lambda^2(\eps) t, \quad \lambda(\eps) >0.\]
In particular, as $r \rightarrow \infty$,
\begin{itemize}
\item $ \psi(r;\eps) < C r^{-\delta}$,  for $C \in \R$ and $\delta \in (0,1)$;  and
\item $ \lambda(\eps)^2 \sim 4\,C(\eps) \,\rme^{-2\gamma_e} \exp\left( \frac{ 2 }{ \eps M}\right)$,  where  $C(\eps) $ represents a constant that depends on $\eps$,  and $\gamma_e$ is the Euler constant.
\end{itemize}
In addition, these target pattern solutions have the following asymptotic expansion for their wavenumber
 \[ k( \eps) \sim  \exp\left (\frac{1}{\eps M} \right ) + \rmO\left( \frac{1}{r^{\delta+1}}\right)\quad \mbox{as} \quad r \rightarrow \infty.\]
\end{Theorem}
\begin{Remark} \label{r:weighted}
If $g \in L^2_\sigma(\mathbb{R}^2)$ with $\sigma > 1$, it follows that 
$$
\int_{\R^2} |g| (1+|{\bf x}|^2)^{(\sigma-1)/3} d{\bf x} \leq \left[\int_{\R^2} |g|^2 (1+|{\bf x}|^2)^{\sigma} d{\bf x} \, \int_{\R^2} (1+|{\bf x}|^2)^{-1-(\sigma-1)/3} d{\bf x} \right]^{1/2} < \infty,
$$
and 
$$
\int_{\R^2} |g|^{2\sigma/(\sigma+1)} d{\bf x} \leq \left[\int_{\R^2} |g|^2 (1+|{\bf x}|^2)^{\sigma} d{\bf x}\right]^{\sigma/(\sigma+1)} \left[ \int_{\R^2} (1+|{\bf x}|^2)^{-\sigma^2} d{\bf x} \right]^{1/(\sigma+1)} < \infty.
$$
Consequently, there is $\delta > 0$ such that $\int_{\R^2} |g({\bf x})|(1+|{\bf x}|^\delta) d{\bf x} < \infty$ and $\int_{\R^2} |g({\bf x})|^{1+\delta} d{\bf x} < \infty$, so that the Schr\"odinger operator $-\Delta + \eps g$ satisfies the hypothesis in~\cite{simon1976bound}.
\end{Remark}

\begin{Remark*}
Here and henceforth in the paper the function $\chi(x) \in C^{\infty}(\R)$ is a cut off function, whose precise form is immaterial, and satisfies $\chi(x) =1$ for $|x| >2$ and $\chi(x) = 0 $ for $|x|<1$.
\end{Remark*}

The rest of this paper is organized as follows: In Section~\ref{s:eikonal} we analyze the case with local coupling  and show the existence of traveling waves for the viscous eikonal equation using matched asymptotics. In Section~\ref{s:weightedspaces} we review properties of Kondratiev spaces and state Fredholm properties of the Laplacian and related operators, leaving the proofs of these results for the appendices. Finally, in Section~\ref{s:nonlocal} we derive Fredholm properties for the convolution operator $\mathcal{L}$ and then, guided by the results from Section~\ref{s:eikonal}, we proceed to prove Theorem~\ref{t:main}. We present a formal derivation of the nonlocal eikonal equation~\eqref{e:nonloc1} in Appendix~\ref{sec:hierarchy}. In Appendix~\ref{sec:proofs}, we prove various subsidiary results that are needed for the proof of Theorem~\ref{t:main}.

\section{Matched asymptotics for 2D Target patterns}\label{s:eikonal}

As we discussed in the introduction, the viscous eikonal equation 
\begin{equation}\label{e:eik}
 \partial_t \phi = \Delta\phi  - | \nabla \phi|^2 + \eps g(x,y) ,\quad (x,y) \in \R^2.
 \end{equation}
is  an abstract model for the evolution of the phase of an array of oscillators with nearest neighbor coupling \cite{doelman2005dynamics}. The perturbation $\eps g(x,y)$, a localized function, represents a small patch of oscillators with a different frequency than the rest of the network. It is well known that this system can produce target patterns  that bifurcate from the steady state, $\phi=0$,  when the parameter  $\eps$ is of the appropriate sign. Our aim in this section is to 
determine an accurate approximation to these  target wave solutions using a formal approach based on matched asymptotics.

We therefore consider solutions to equation \eqref{e:eik} of the form $\phi(x,y,t) = \tilde{\phi}(x,y) - \omega t$, where $\tilde{\phi}$ solves
\begin{equation}\label{e:eikwaves}
-\omega = \Delta \tilde{\phi}  - |\nabla \tilde{\phi}|^2 + \eps g(x,y) ,\quad (x,y) \in \R^2, \quad \omega>0.
 \end{equation}
We also assume in this section that $g$ is a radial and algebraically localized function that satisfies 
\begin{equation}
\int_0^\infty |g(r)| (1+r)^\sigma r dr < \infty, \quad \mbox{for some} \quad \sigma>0.
\label{assumption}
\end{equation}
This simplifies our analysis since we can restrict ourselves to finding radially symmetric solutions. In addition, because the viscous eikonal equation~\eqref{e:eikwaves} only depends on derivatives of $\tilde{\phi}$, we can recast it as a first order ODE for the wavenumber $\zeta  = \tilde{\phi}_r$:
\begin{equation} 
\zeta_r + \frac{\zeta}{r} - \zeta^2 + \eps g(r) = -\omega.
\label{eq:wavenumber}
\end{equation}
Here, $\omega = \omega(\eps)$ is an eigenparameter, i.e. it is not specified, rather it is determined in such a way as to ensure that the solutions satisfy the required boundary conditions. 

For radial solutions that are regular at the origin, $\zeta(r)$ is $\rmO(r)$ as $r \to 0$, and we can rewrite the ODE in an equivalent integral form
\begin{equation}
\zeta(r) = \frac{1}{r}\int_0^r \left[\zeta^2(\eta) - \omega - \eps g(\eta) \right] \eta d \eta.
\label{eq:integral}
\end{equation}
Moreover, because we are bifurcating from the trivial state $\zeta = 0$, we can assume that $\zeta$ is small if $\eps$ is small, and posit  the following regular expansions
 \begin{align*}
 \zeta & = \eps \zeta_1 + \eps^2 \zeta_2+ \eps^3 \zeta_3 + \cdots,\\
 \omega& = \eps \omega_1 + \eps^2 \omega^2 + \eps^3 \omega_3 + \cdots.
 \end{align*} 
 Substituting these expressions in~\eqref{eq:integral}, we obtain that at $\rmO(\eps)$,
 \[ 
 \zeta_1 = -\frac{\omega_1 r}{2} - \frac{1}{r} \int_0^r g(\eta) \; \eta d\eta. 
 \]

 Now, because we are interested in target patterns, solutions should satisfy  $\zeta \rightarrow k(\eps)>0$ as $r \rightarrow \infty$, where $k(\eps)$ is the asymptotic wavenumber.  This requires us to consider functions $\zeta_1$ that have a finite limit as $r \to \infty$, and forces us to pick $\omega_1 = 0$. The result is that   $r \zeta_1(r) \to -\int g(r)\;rdr = -M < \infty$, in the limit of $r$ going to infinity.
 
 At the same time, from  assumption~\eqref{assumption} on the algebraic localization of $g$, we have the quantitative estimate
$$
|M - r \zeta_1(r)| \leq \frac{1}{(1+r)^{\sigma}} \int_r^\infty |g(\eta)| (1+\eta)^{\sigma} \eta d \eta \leq C (1+r)^{-\sigma},
$$
so that at order $\rmO(\eps^2)$ we find 
$$
\zeta_2 = -\frac{\omega_2 r}{2} + \frac{1}{r} \int_0^r \zeta_1^2 \; \eta d\eta,
$$
with $\zeta_1^2$ in $L^1(\mathbb{R}^2)$. Again the boundary conditions force $\omega_2 = 0$, and as a result we obtain  $\displaystyle{\zeta_2 \sim \frac{M^2}{r} \ln \left(\frac{r}{r_c}\right)}$ for large values of $r$. Here the constant $r_c$ depends on $g$ and is given by the following limit
$$
\ln(r_c) = \lim_{r \to \infty}  \left[\ln(r) - \frac{1}{M^2}  \int_0^r \zeta_1^2 \; \eta d\eta\right],
$$
which we know exists from the estimate for $|M - r \zeta_1(r)|$. 

Note that these expressions for $\omega_2$ and $\zeta_2$ imply that the asymptotic wavenumber is $\rmo(\eps^2)$.  In fact, continuing this procedure it is easy to check that  we get $\omega_j = 0$ for all $j$, so that the frequency $\omega$ is $\rmo(\eps^n)$ for all orders in $\eps$. In addition, the expressions for $\zeta_1$ and $\zeta_2$ yield the expansion
\begin{equation}
\zeta \approx -\frac{\eps M}{r} +  \frac{\eps^2 M^2 \ln(r/r_c)}{r} + \cdots.
\label{inner-expnsn}
\end{equation}
whose terms are not uniformly ordered. For instance,
\[ |\eps^2 \zeta_2| \geq | \eps \zeta_1 | \quad \mbox{for} \quad r \geq r_c \exp\left(\frac{1}{\eps |M|} \right).\]
This suggests that the above {\it inner} expansion  is not uniformly valid. We therefore need  to introduce an {\it outer} expansion and match both solution in an intermediate region given by $r \sim r_c \exp(|\eps^{-1} M^{-1}|)$.

 In this intermediate region the inhomogeneity, $\eps g(r)$, and the frequency, $\omega$, are small compared to the other terms in the equation, so that the radial eikonal equation~\eqref{eq:wavenumber} reduces to 
\[\zeta_r + \frac{\zeta}{r} - \zeta^2 \approx 0.\]
We can solve this explicitly to find that
\[ \zeta = \frac{1}{r( C - \ln(r))},\]
where $C = C(\eps)$ is a, {\em possibly $\eps$ dependent}, constant of integration. Comparing this result with the inner expansion~\eqref{inner-expnsn} leads to $C \approx \frac{-1}{\eps M}$, to leading order. We can then write $C = -(\eps M)^{-1} + c_0 + c_1 \eps + \ldots$, and for fixed $r$ and as $\eps \to 0$, obtain
\[ \zeta = \frac{\eps M }{r} + \frac{\eps^2 M^2(\ln(r)-c_0)}{r} + \cdots.\]
Comparing again with the inner expansion~\eqref{inner-expnsn}, we see that $r_c = \exp(c_0)$.

In the outer region, where we retain the frequency, $\omega$, and neglect the inhomogeneity, g,  solutions are described by the equation
\[ \zeta_r +\frac{\zeta}{r} - \zeta^2 = - \omega.\]
If we define the 'outer' variable $\xi = \sqrt{\omega} r$ and scaling function $F$ so that $\zeta(r) = \sqrt{\omega}F( \sqrt{\omega}r)$, then $F$ satisfies the $\omega$ (and hence also $\eps$) independent equation $F_\xi +F/\xi -F^2 = -1$. Using the (differentiated) Hopf-Cole transformation $F(\xi) = - \psi'(\xi)/\psi(\xi)$, and then solving for $\psi$ gives
\[ F(\xi) = - \frac{\partial_\xi K_0(\xi)}{K_0(\xi)},\]
 where $K_0(\xi)$ is the modified Bessel's function of the first kind \cite{whittaker1996course}. Consequently, for $\xi = \sqrt{\omega} r \ll 1$ fixed and $\eps \to 0$, a solution $\zeta$ of the outer equation  is given by 
\begin{equation} 
\zeta(r) \sim \frac{1}{r( - \ln( \sqrt{\omega r}/2) - \gamma_e)}, 
\label{outer-expnsn}
\end{equation}
where $\gamma_e = 0.5772\ldots$ is the Euler constant \cite{Abram_Stegun}. This approximation is also valid in the intermediate region, allowing us to match it to the inner expansion, 
\[ \frac{1}{r(-\frac{1}{\eps M} + \ln(r_c) -\ln(r))} \approx \frac{1}{r( - \frac{1}{2} \ln \omega + \ln 2 - \gamma_e - \ln(r)) },\]
and obtain the following approximation for the frequency 
\[ \omega \sim \frac{4 \exp(-2 \gamma_e)}{r_c^2}\exp\left(\frac{2}{\eps M}\right). \]
Hence, if we assume $\eps M < 0$, this does indeed show that $\omega$ is small beyond all orders in $\eps$.

\begin{table}
 \caption{In this section we will make use of the asymptotic behavior of $K_0(z)$ and its derivative $\partial_z K_0(z) = - K_1(z)$ (see Ref.~\cite{Abram_Stegun}), which we summarize in this table
}
\begin{center}
\begin{tabular}{ c  c  c}
\hline
& $ z \rightarrow 0$ & $z \rightarrow \infty$\\
\hline\\
$K_0(z)$ & $-\ln (z/2) - \gamma_e + \rmO(z^2 |\ln z|)$ &$ \sqrt{\frac{\pi}{2z}}e^{-z} (1+ \rmo(1/z))$\\[2ex]
 $K_1(z)$ & $\frac{1}{z}+\rmO(z |\ln z|)$ &$ \sqrt{\frac{\pi}{2z}}e^{-z} (1+ \rmo(1/z))$\\[2ex]
\hline
  \end{tabular} 

  \end{center}
\label{tab:Bessel} 
  \end{table}

\begin{Remark*}
The viscous eikonal equation~\eqref{e:eik} is conjugate to a Schr\"odinger eigenvalue problem via the Hopf-Cole transform. The frequency $\omega =  -\partial_t \phi$ in the viscous eikonal equation corresponds to the ground state energy for the Schr\"odinger operator $\Delta - \eps g(r)$. 
Indeed, the expression above is a refinement of the results from \cite{simon1976bound} and \cite{kollar2007coherent} for the ground state eigenvalue/frequency respectively,  in that we have an expression for the numerical prefactor. 
\end{Remark*}

\begin{Remark*} Integrating $\zeta = \partial_r \phi$, we can determine the phase $\phi(r)$ for the target patterns. From~\eqref{inner-expnsn}~and~\eqref{outer-expnsn} we get the inner and outer expansions for the phase $\phi$ 
\[
\phi(r) \sim \phi_0 + \begin{cases} - \ln\left[ \frac{-1}{\eps M} + c_0 + c_1 \eps + \cdots - \ln(r)\right] & r \text{ fixed}, \eps \to 0 \\ - \ln(K_0(\sqrt{\omega} r)) &  \xi = \sqrt{\omega} r  \text{ fixed}, \eps \to 0 \end{cases}
\]
where $\phi_0$ is an arbitrary phase shift corresponding to a constant of integration.
\end{Remark*}

\begin{figure}\label{f:simulations}
    \centering
    \begin{minipage}{0.8\textwidth}
    \centering
    \begin{subfigure}{0.5\textwidth}
    \centering
    \includegraphics[width=1\textwidth]{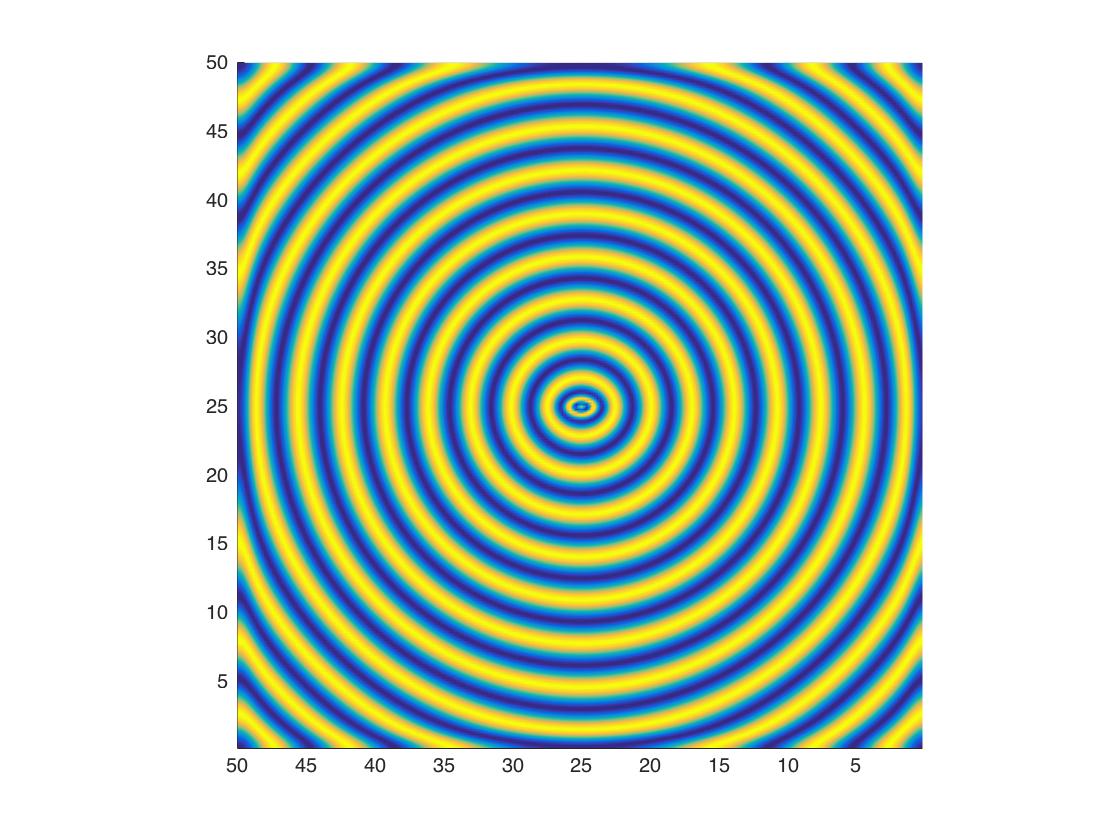}
	\caption{}
	\label{f:ellipse-out}
    \end{subfigure}%
    \begin{subfigure}{0.5\textwidth}
    \centering
    \includegraphics[width=1\textwidth]{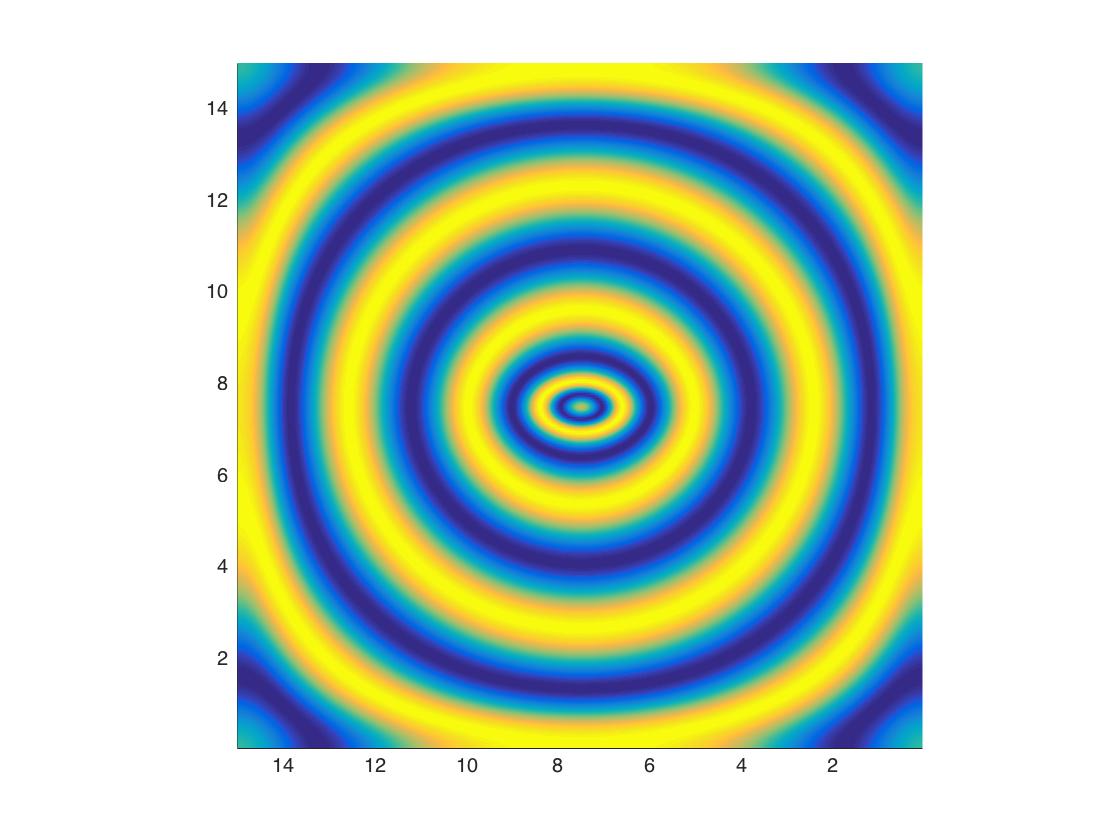}
    \caption{}
    \label{f:ellipse-in}
    \end{subfigure}%
    
\end{minipage}
\centering
\begin{minipage}{0.8\textwidth}
    \centering
    \begin{subfigure}{0.5\textwidth}
    \centering
    \includegraphics[width=1\textwidth]{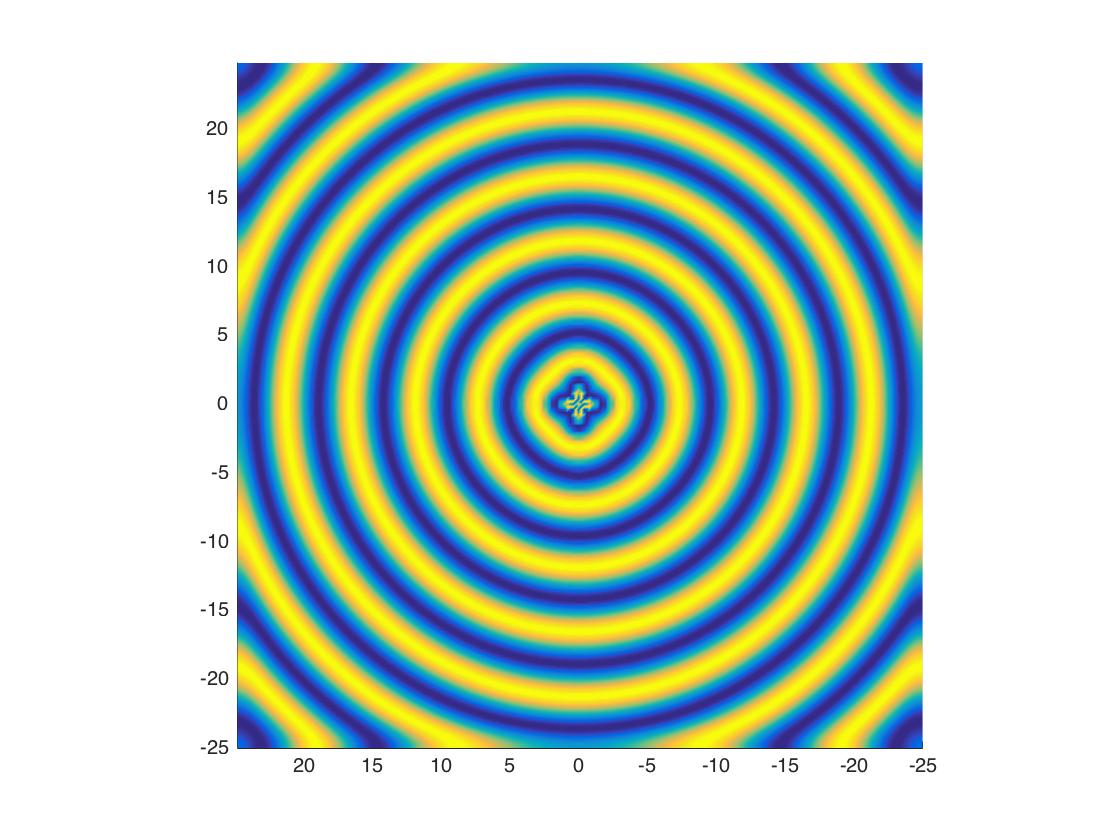}
	\caption{}
	\label{f:square-out}
    \end{subfigure}%
    \begin{subfigure}{0.5\textwidth}
    \centering
    \includegraphics[width=1\textwidth]{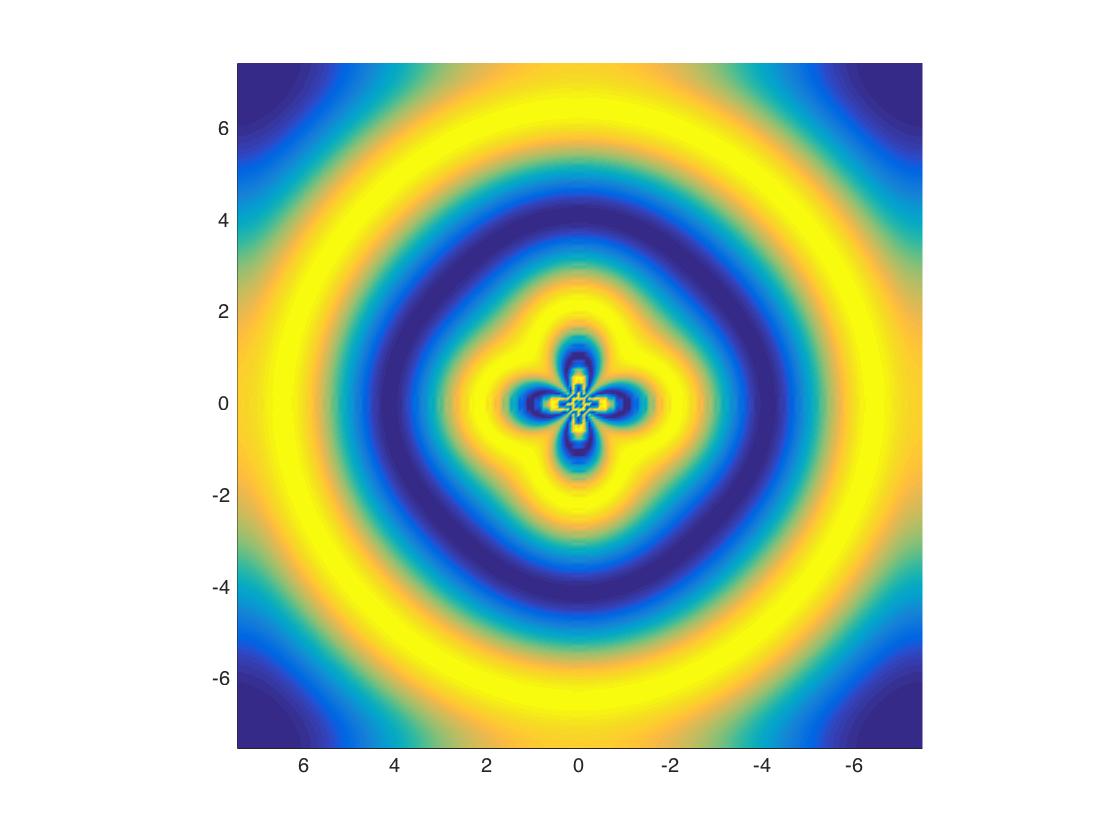}
	\caption{}
	\label{f:square-in}
    \end{subfigure}%
    
    \end{minipage}
\caption{ $g(x,y)  = \displaystyle \frac{1}{(1 + 3x^2+y^2)^{3/2}}$ in Figures~(a)~and~(b). Notice how the pattern in the far field is radially symmetric, Figure (a), where as near the inhomogeneity it is elliptical, Figure (b). On the other hand, Figures (c) and (d) correspond to a system perturbed by  $g(r,\theta) = \displaystyle \frac{1 + \cos(4\theta)}{(1+r)^3}$. Again notice how the pattern is radially symmetric in the far field, Figure (c), even though the inhomogeneity, $g$, is not radially symmetric, which results in a nonradially symmetric core, Figure (d).}
    \label{f:numerics}
\end{figure}

The matching procedure above relied strongly on the coupling kernels $\mathcal{L}$ and $J$ being local, and having a radial inhomogeneity $g(r)$, so the results do not immediately carry over to  the case of nonlocal coupling and/or nonradial and algebraically localized inhomogeneities. Figure~\ref{f:numerics} depicts numerical results for particular cases of the nonlocal eikonal equation~\eqref{e:nonloc1} given by 
$$
\partial_t \phi = (\Id - \Delta)^{-1} \Delta \phi - | \nabla (\Id-\Delta)^{-1} \phi|^2 + \eps g(x,y), \quad (x,y) \in \R^2,
$$
with $\eps  < 0$ and $g(x,y) = (1 + 3x^2+y^2)^{-3/2}$ in \ref{f:numerics}(a)~and~(b) (respectively $g(r,\theta) = (1 + \cos(4\theta))/ (1+r)^3$ in \ref{f:numerics}(c)~and~(d)). This evolution equation was integrated using a spectral discretization for the spatial operator and  exponential time differencing  (ETD) for the time stepping \cite{cox2002exponential,kassam2005fourth}.  We will present a full discussion of our numerical methods and results 
in future work;  here we only note that the far-field behavior of the target waves are (nearly) radially symmetric (see Figure~\ref{f:numerics}) even in the general nonlocal problem. Indeed, setting $\partial_t \phi = - \omega$ with  $\omega > 0$, and rescaling to the ``outer variables" $\tilde{x} = \sqrt{\omega} x, \tilde{y} = \sqrt{\omega} y$, we get  $\tilde{\nabla} = \omega^{-1/2} \nabla, \tilde{\Delta} = \omega^{-1} \Delta$. Further, from the algebraic localization of $g$, it follows that 
$$
\frac{1}{\omega} g(\omega^{-1/2}\tilde{x}, \omega^{-1/2} \tilde{y}) \sim \rmO(\omega^\sigma) \to 0  \text{ as } \omega \to 0.
$$
Consequently, the (formal) $\eps \to 0$  limit equation in the outer variables is the viscous eikonal equation
$$
- 1 = (\Id - \omega \tilde{\Delta})^{-1} \tilde{\Delta} \phi - | \tilde{\nabla} (\Id-\omega \tilde{\Delta})^{-1} \phi|^2 + \eps \omega^{-1} g(\omega^{-1/2}\tilde{x}, \omega^{-1/2} \tilde{y}) \approx \tilde{\Delta} \phi - |\tilde{\nabla} \phi|^2.
$$

This argument, together with the numerical results depicted in Figure~\ref{f:numerics}, suggests that, even for general inhomogeneities, we may approximate the dynamics of target waves solutions in the outer region by a radial viscous eikonal equation
\[ -\omega =  \partial_{rr} \phi  + \frac{1}{r} \partial_r \phi- (\partial_r \phi)^2.\]
This intuition will guide our analysis of the nonlocal equation~\eqref{e:nonloc1} for general coupling kernels and nonradially symmetric, algebraically localized perturbations $g$. We will show that the frequency indeed scales as $\omega \sim \exp \left( \frac{2}{\eps M} \right)$, with $\eps M<0$,  and that the target waves that bifurcate from the steady state are radially symmetric far away from the inhomogeneity. Our strategy consists of first finding, in Section~\ref{s:nonlocalintermediate}, solutions to equation~\eqref{e:nonloc1}, with $\omega=0$, which give the appropriate intermediate approximation. Then in Section~\ref{s:nonlocalfull} we find the ``outer" solutions to equation~\eqref{e:nonloc1} by treating the frequency $\omega$ as an extra parameter. Finally, in Section~\ref{s:nonlocalmatching} we derive a relation between the frequency $\omega$ and the parameter $\eps$ using asymptotic matching, and then proceed to prove the results of Theorem \ref{t:main}.

\section{Weighted Spaces}\label{s:weightedspaces}
We define the Kondratiev space \cite{mcowen1979behavior}, $M^{s,p}_\gamma(\R^d)$, with $d \in \N, s\in \N \bigcup \{0\},  \gamma \in \R,  p \in (1,\infty)$,  as the space of locally summable, $s$ times weakly differentiable functions $u: \R^d \rightarrow \R$ endowed with the norm
\[ \| u\|^p_{M^{s,p}_\gamma(\R^d)} = \sum_{|\alpha|\leq s} \left \| D^{\alpha}u({\bf x}) \cdot (1+|{\bf x}|^2)^{(\gamma + |\alpha|)/2} \right \|^p_{L^p(\R^d)}. \]
From the definition, it is clear that these spaces admit functions with algebraic decay or growth, depending on the weight $\gamma$, and that these functions gain localization with each derivative. Moreover, given real numbers $\alpha, \beta$, such that $\alpha>\beta$, the embedding $M^{s,p}_\alpha (\R^d)\subset M^{s,p}_\beta(\R^d)$ holds, and additionally if $s$ and $r$ are integers such that $s<r$ then  $M^{s,p}_\gamma(\R^d) \subset M^{r,p}_\gamma(\R^d)$. As in the case of Sobolev spaces, we may identify the dual $(M^{s,p}_\gamma(\R^d))^*$ with the space $M^{-s,q}_{-\gamma}(\R^d)$, where $p$ and $q$ are conjugate exponents, and in the case when $p=2$ we also have that Kondratiev spaces are Hilbert spaces. In particular, given $f, g \in M^{s,2}_\gamma(\R^d)$ the pairing 
$$\langle f, g\rangle := \sum_{|\alpha|\leq s} \int_{\R^d} D^\alpha f({\bf x}) \cdot D^{\alpha} g({\bf x}) \cdot (1+ |{\bf x}|^2)^{(\gamma+|\alpha|)} \;dx $$  
satisfies all the properties of an inner product. This is not hard to see once we notice that for every $f \in L^2_\gamma(\R^d)$ the function $f({\bf x}) \dot (1+|{\bf x}|^2)^{\gamma/2}$ is in the familiar Hilbert space $L^2(\R^d)$.

We will use this last property to decompose the space $M^{s,2}_\gamma(\R^d)$ into a direct sum of its polar modes. Here we restrict ourselves to the case $d=2$ which is relevant for our analysis, but mention that a similar decomposition is possible in higher dimensions. More precisely, using  the notation $M^{s,p}_{r, \gamma}(\R^d) $ to denote the subspace of radially symmetric functions in $M^{s,p}_\gamma(\R^d)$, we show that
\begin{Lemma}\label{l:polar}
Given $s \in \N \bigcup \{ 0\}$ and $\gamma \in \R$,  the space  $M^{s,2}_{\gamma}(\R^2)$ can be written as a direct sum decomposition
 \[ M^{s,2}_{\gamma}(\R^2) =  \bigoplus m^n_\gamma,\]
  where $n \in \Z$ and
\[ m^n_\gamma = \left \{ u \in M^{s,2}_\gamma(\R^2) \mid u = w(r) \rme^{\rmi n \theta} \; \mbox{and} \; w(r) \in M^{s,2}_{r,\gamma} (\R^2) \right \}.\]
\end{Lemma}
The proof of this result follows the analysis of Stein and Weiss in \cite{stein2016introduction}.
\begin{proof} We need to show that each element $f \in M^{s,p}_\gamma(\R^2)$ can be well approximated by an element in the direct sum $\bigoplus m^n_\gamma$. To obtain a candidate function in the latter space we first identify $\R^2$ with the complex plane $\C$ by letting $z = x +\rmi y$ for $(x,y) \in \R^2$. Then, using the notation $z= r\rme^{\rmi \theta}$ we write  $f(z) = f(r\rme^{\rmi \theta})$  for each function $f \in M^{s,2}_\gamma(\R^2)$. By Fubini's Theorem the function $f(r \rme^{\rmi \theta})$ is in $H^s([0,2\pi])$ for a.e. $r \in [0,\infty)$, so we may express this function as a Fourier series in $\theta$.
\[ f(r\rme^{\rmi \theta} ) = \sum_{n \in \Z} f_n(r) \rme^{\rmi n \theta}, \quad \mbox{where} \quad f_n(r) = \frac{1}{2\pi} \int_0^{2\pi} f(r\rme^{\rmi \theta} ) \cdot \rme^{-\rmi n \theta} d\theta \]
Notice that the functions $f_n(r)$ are in the space $ M^{s,2}_{r,\gamma}(\R^2)$, so that this sum is the desired candidate function. Because the series $\sum_{-N}^N |f_n(r)|^2$ is monotonically increasing and because by Parseval's identity it converges to $\frac{1}{2 \pi} \int_0^{2\pi} | f(r\rme^{\rmi \theta})|^2 d\theta$, letting $g_N(r) = 2 \pi \sum_{-N}^N |f_n(r)|^2$ a straight forward calculation shows that 
\[ \int_0^\infty \left( \int_0^{2\pi} |f - g_N|^2 \;d\theta \right) \; (1 +r^2)^{\gamma} rdr = \int_0^\infty \left [ \left ( \int_0^{2\pi} |f(r\rme^{\rmi \theta})|^2 \;d\theta \right) - g_N \right] \;(1+r^2)^\gamma rdr.\]
Then, by the monotone convergence theorem we may conclude 
\[ 
\left \| f(r\rme^{\rmi \theta}) - 2 \pi \sum_{-N}^N f_n(r) \rme^{\rmi n \theta} \right \|_{L^2_\gamma(\R^2)} = \| f - g_N\|_{L^2_\gamma (\R^2)} \rightarrow 0 , \quad \mbox{as} \quad N \rightarrow \infty,\]
as desired. Moreover, since for a.e. $r \in [0,\infty)$ the function $f \in H^s([0,2 \pi])$, a  similar argument can be carried out to show that as $N \rightarrow \infty $ the expressions  $\left \| \partial_\theta^\alpha f(r\rme^{\rmi \theta}) - 2 \pi \sum_{-N}^N (in)^\alpha f_n(r) \rme^{\rmi n \theta} \right \|_{L^2_\gamma(\R^2)} \rightarrow 0 $  for all integers $\alpha \leq s$. This completes the proof of the lemma.
\end{proof}

We also have the following result describing how elements in $M^{1,2}_{\gamma}(\R^d)$ decay at infinity (see Appendix \ref{s:appendixdecay} for a proof).
\begin{Lemma}\label{l:decay}
Given $f \in M^{1,2}_{\gamma}(\R^d)$, then $|f({\bf x})| \leq C\|f\|^{d/2}_{M^{1,2}_{\gamma}(\R^d)} \cdot (1+|{\bf x}|^2 )^{-(\gamma+d/2)}$ as ${\bf x} \rightarrow \infty$.
\end{Lemma}

In addition, the next lemma characterizes the multiplication property for Kondratiev spaces. The lemma is more general than we need in the sense that it holds for complete Riemannian manifolds that are euclidean at infinity, $(\mathbb{M},e)$. A proof of this result can be found in \cite{choquet1981elliptic}. We have adapted the notation so that it is consistent with our definition of Kondratiev spaces.

\begin{Lemma}\label{l:multiplication}
If $(\mathbb{M},e)$ is a complete Riemannian manifold euclidean at infinity of dimension $d$, we have the continuous multiplication property 
\begin{align*}
 M^{s_1,2}_{\gamma_1}(\mathbb{M}) \times M^{s_2,2}_{\gamma_2}(\mathbb{M}) &\longrightarrow M^{s,2}_\gamma(\mathbb{M})\\
 (f_1,f_2) &\longmapsto f_1\cdot f_2
 \end{align*}
 
provided $s_1,s_2 \geq s, \; s< s_1 + s_1 - d/2,$ and $ \gamma< \gamma_1 + \gamma_2+d/2$.
\end{Lemma}

\subsection{Fredholm properties of the Laplacian and related operators}\label{s:fredholmProperties}

The main appeal of Kondratiev spaces for us is that the the Laplace operator is a Fredholm operator in these spaces. This is summarized in the following theorem, whose proof can be found in \cite{mcowen1979behavior}. This result is the basis for deriving Fredholm properties for other linear operators that will be encountered in Section \ref{s:nonlocal}.

\begin{Theorem}\label{McOwen}
Let $1<p=\frac{q}{q-1}<\infty$,  $d \geq 2$, and $\gamma \neq d/q + m $ or $\gamma \neq 2-d/p -m $, for some $m \in \N$.  Then 
\[ 
\Delta: M^{2,p}_{\gamma-2}(\R^d) \rightarrow L^p_{\gamma}(\R^d), 
\]
is a Fredholm operator and
\begin{enumerate}
\item for $2-d/p < \gamma < + d/q$ the map is an isomorphism;
\item for $ d/q + m < \gamma<  d/q + m+1$ , $m \in \N$, the map is injective with closed range equal to 
\[
R_m = \left\{ f \in L^p_{\gamma } : \int f(y)H(y) =0 \  \text{for all } \ H \in \bigcup_{j=0}^m \mmH_j\right\}; 
 \]
\item for $2-d/p - m -1 < \gamma <2-d/p -m $, $m \in \N$, the map is surjective with kernel equal to 
\[ N_m = \bigcup_{j=0}^m \mmH_j
.\]
\end{enumerate}
Here,  $\mmH_j $ denote the harmonic homogeneous polynomials of degree $j$.

On the other hand, if $\gamma = 2-d/p - m $ or $\gamma =   d/q + m$ for some $m \in \N$, then $\Delta$ does not have closed range.
 \end{Theorem}

Notice that we can use the result from Lemma \ref{l:polar} to diagonalize the Laplacian. That is, given $u \in M^{2,2}_{\gamma-2}( \R^2)$, 
\[ \Delta u = \Delta \left (\sum_{n \in \Z} u_n(r) \rme^{\rmi n \theta} \right)=  \sum_{n \in \Z}   \left( \partial_{rr} u_n + \frac{1}{r} \partial_r u_n - \frac{n^2}{r^2}  u_n \right) \rme^{\rmi n \theta} = \sum_{n \in \Z} ( \Delta_n u_n)  \rme^{\rmi n \theta},\]
where $u_n \in M^{2,2}_{r,\gamma-2}(\R^2)$. 

One can now combine this decomposition together with Theorem \ref{McOwen} to arrive at the following lemma.

\begin{Lemma}\label{l:Delta_nFredholm}
Let $\gamma \in \R \backslash \Z$, and $n \in \Z$. Then, the operator $\Delta_n: M^{2,2}_{r, \gamma-2}(\R^2) \rightarrow L^2_{r,\gamma}(\R^2)$ given by 
\[ \Delta_n \phi = \partial_{rr} \phi + \frac{1}{r} \partial_r \phi - \frac{n^2}{r^2} \phi\]
is a Fredholm operator and,
\begin{enumerate}
\item for $1-n < \gamma< n+1$, the map is invertible;
\item for $\gamma>n+1 $, the map is injective with cokernel spanned by $r^n$;
\item for $\gamma< 1-n$, the map is surjective with kernel spanned by $r^n$.
\end{enumerate}
On the other hand, the operator is not Fredholm for integer values of $\gamma$.
\end{Lemma}

The next lemma requires that we specify some notation. In this paper we use the symbol $W^{s,p}_{\sigma}$ to denote the space of locally summable, $s$ times weakly differentiable functions $u:\R^d \rightarrow \R$ endowed with the norm
\[ \| u\|^p_{W^{s,p}_\sigma(\R^d)} = \sum_{|\alpha|\leq s} \left \| D^{\alpha}u({\bf x}) \cdot (1+|{\bf x}|^2)^{\sigma/2} \right \|^p_{L^p(\R^d)}. \]
Then, we write $W^{s,p}_{r,\sigma}$ to denote the subspace of radially symmetric functions in $W^{s,p}_{\sigma}$. With this notation we can summarize the Fredholm properties of the operator $\mathscr{L}_\lambda$, defined in the next lemma.

\begin{Lemma}\label{l:fredholmLambda}
Given $\gamma \in \R$, $\lambda \in [0,\infty)$, and $p \in (1,\infty)$, the operator $\mathscr{L}_\lambda: \mathcal{D} \longrightarrow L^p_{r,\gamma}(\R^2)$ defined by 
\[ \mathscr{L}_\lambda \Phi = \partial_{rr} \Phi + \frac{1}{r} \partial_r \Phi - 2\lambda \partial_r \Phi \]
 and with domain $\mathscr{D}=\{ u \in M^{1,p}_{r,\gamma-1}(\R^2) \mid \partial_r u \in W^{1,p}_{r,\gamma}(\R^2) \}$, is Fredholm for $\gamma \neq 1-2/p$. Moreover,
\begin{itemize}
\item it is invertible for $\gamma>1-2/p$, and
\item it is surjective with $\ker =\{1\}$ for $\gamma<1-2/p$.
\end{itemize}
\end{Lemma}
\begin{proof}
The result follows from Proposition \ref{p:dr} and  Lemma \ref{l:invertible} in Appendix \ref{s:appendixRadialderivatives}, which show that the operator $\mathcal{L}_\omega  = (\partial_r + \frac{1}{r} - 2\lambda) \partial_r$ is the composition of a Fredholm index $i =0$ operator, $\partial_r : M^{1,p}_{r,\gamma-1}(\R^2) \rightarrow M^{1,p}_{r,\gamma}(\R^2)$, and an invertible operator, $ (\partial_r + \frac{1}{r} - 2\lambda): W^{1,p}_{r,\gamma}(\R^2) \rightarrow L^p_{r,\gamma}(\R^2)$. 
\end{proof}

Lastly, we include the following proposition whose proof can be found in Appendix \ref{s:appendixb}.
\begin{Proposition}\label{p:fredholmNormalF}
Let $\gamma \in \R$, $\alpha, \beta \in \Z^+ \bigcup \{ 0\}$, $m,d \in \Z^+$, and $l \in \Z$. Then, the operator, 
\[\Delta^m(\Id-\Delta)^{-l}:  \mathcal{D} \subset L^p_{\gamma- 2m} (\R^d)  \longrightarrow  L^p_\gamma(\R^d),\]
with domain $\mathcal{D}=\{u \in L^p_{\gamma-2m}(\R^d) \mid (Id-\Delta)^{-l}u \in M^{2m,p}_{\gamma-2m}(\R^d) \}$  
\begin{itemize}
\item is a Fredholm operator for $\alpha+d/p<\gamma< -\beta -d/p +2m$  with kernel and cokernel given by
\[ \ker = \bigcup_{j=0}^\beta \mathcal{H}_{j,k}, \quad  \coker = \bigcup_{j=0}^\alpha \mathcal{H}_{j,k};\]
 \item and not Fredholm for values of  $\gamma \in \{ j + d/p : j \in \Z\}$. 
\end{itemize}
\end{Proposition}

\section{Nonlocal Eikonal Equations}\label{s:nonlocal}

At the outset, we recall our model nonlocal eikonal equation~\eqref{e:nonloc1} and the hypotheses on the coupling kernels. Our goal for this section is to show the existence of target wave solutions for~\eqref{e:nonloc1} which bifurcate from the spatially homogeneous solution. We concentrate on the equation
$$
 \phi_t = \mathcal{L}\ast \phi - | \mathcal{J} \ast\nabla  \phi|^2 + \eps g(x,y), \quad (x,y) \in \R^2, 
 $$
where again we assume $\mathcal{L}$ is a diffusive kernel that commutes with rotations. As a consequence its Fourier symbol, $L({\bf k})$, is real analytic and a radial function. We also recall the following assumptions.

 \begin{Hypothesis*}{\bf \ref{h:analyticity} }
  The  multiplication operator $L$ is a function of $\xi := | {\bf k} |^2$. Its domain can be extended to a strip in the complex plane, $\Omega = \R \times (-\rmi \xi_0, \rmi \xi_0)$ for some sufficiently small and positive $\xi_0 \in \R$, and on this domain the operator is uniformly bounded and analytic. Moreover, there is a constant $\xi_m \in \R$ such that the operator $L(\xi)$ is invertible with uniform bounds for $|\Re \xi| > \xi_m$.
  \end{Hypothesis*}
 
 Note that, because $L(\xi)$ is analytic, its zeros are isolated.
 
 \begin{Hypothesis*}{\bf \ref{h:multiplicity}}
 The multiplication operator $L(\xi)$ has a zero, $\xi^*$, of multiplicity $\ell$ which we assume is at the origin. Therefore, the symbol $L(\xi)$ admits the following Taylor expansion near the origin.
  \[ L(\xi ) = (- \xi)^\ell + \rmO(\xi^{\ell+1}), \quad \mbox{for} \quad \xi \sim 0.\]
  In particular, we pick $\ell=1$.
 \end{Hypothesis*}
 
\begin{Hypothesis*} {\bf \ref{h:nonlinearity}}
The kernel $\mathcal{J}$ is radially symmetric, exponentially localized, twice continuously differentiable, and 
\[ \int_{\R^2}\mathcal{J}({\bf x})\;d{\bf x}=1.\]
\end{Hypothesis*}

The idea behind our proof for the existence of target wave solutions is to first show that the convolution operator $\mathcal{L}$ behaves much like the Laplacian when viewed in the setting of Kondratiev spaces. In other words, both operators have the same Fredholm properties. Consequently it is possible to precondition equation \eqref{e:nonloc1} by an appropriate operator, $\mathscr{M}$, with average one, and obtain an expression which has the Laplacian as its linear part,
\[ \mathscr{M} \ast \partial_t\phi = \Delta \phi - \mathscr{M} \ast | \mathcal{J} \ast \nabla \phi|^2 + \eps \mathscr{M} \ast g.\]
We then proceed in a similar manner as in Section~\ref{s:eikonal} and look for solutions of the form, $\phi(x,y,t) = \tilde{\phi}(x,y) - \omega t$. Dropping the tildes from our notation we arrive at 
\begin{equation}\label{e:preconditioned}
 -\omega = \Delta \phi - \mathscr{M} \ast |\mathcal{J }\ast \nabla \phi|^2 + \eps \mathscr{M} \ast g.
 \end{equation}
 
We will look at the following two problems. 
\begin{enumerate}
\item Finding a solution to the intermediate approximation, described by equation \eqref{e:preconditioned} with the value of the frequency, $\omega$, equal to zero. 
\item Finding a solution valid on the whole domain described again by equation \eqref{e:preconditioned}, but where we let $\omega = \lambda^2$ be a non-negative parameter.
\end{enumerate} These two solutions are matched, and then the results of Theorem \ref{t:main} are shown.

 This section is organized as follows. In the next subsection we will derive Fredholm properties for the convolution operator $\mathcal{L}$, as well as mapping properties for a number of related convolution operators. Then in Sections \ref{s:nonlocalintermediate} and \ref{s:nonlocalfull} we prove the existence of solutions to the intermediate approximation and to the full problem, respectively. Finally, in Section \ref{s:nonlocalmatching} we prove Theorem \ref{t:main}.

 \subsection{Nonlocal Operators}\label{s:nonlocaloperators}
 
 The following proposition is the 2-dimensional version of the results form \cite{jaramillo2016effect}, but for convolution kernels with radial symmetry. The results below follow very closely the proofs outlined in Ref.~\cite{jaramillo2016effect}, and we include them for the sake of completeness. The proof shows that, with the Hypothesis~\ref{h:analyticity} and the more general version of Hypothesis \ref{h:multiplicity} in which we assume $\ell \in \mathbb{N}$, the convolution operator $\mathcal{L}$ has the same Fredholm properties as the operator $\Delta^\ell$. 
 
\begin{Proposition}\label{p:fredholmconv}
Let $\gamma \in \R$, $\alpha, \beta \in \N \bigcup \{0\}$, and suppose the convolution operator $\mathcal{L}: L^2_{\gamma -2}(\R^2) \rightarrow L^2_{\gamma}(\R^2)$ satisfies Hypothesis \ref{h:analyticity}, and Hypothesis \ref{h:multiplicity} with $\ell \in \N$. Then, with appropriate domain $\mathcal{D}$ and
 
\begin{itemize}
\item for $\alpha+2/p<\gamma< -\beta -2/p +2\ell$,  $\mathcal{L}$ is a Fredholm operator with kernel and cokernel given by
\[ \ker = \bigcup_{j=0}^\beta \mathcal{H}_{j,k}\quad  \coker = \bigcup_{j=0}^\alpha \mathcal{H}_{j,k},\]

 \item whereas for $\gamma = \{ j + 2/p : j \in \Z\}$, $\mathcal{L}$ is not Fredholm.
\end{itemize}

 \end{Proposition}
The above result follows from Proposition \ref{p:fredholmNormalF} and  Lemmas \ref{l:decompositionconv} and \ref{l:isomorphismMLR} shown below.

 \begin{Lemma}\label{l:decompositionconv}
 Let the multiplication operator, $L(\xi)$,  satisfy Hypothesis \ref{h:analyticity}, and Hypothesis \ref{h:multiplicity} with $\ell \in \N$. Then $L(\xi)$ admits the following decomposition:
 \[ L(\xi) = M_L(\xi) L_{NF}(\xi) = L_{NF}(\xi) M_R(\xi),\]
 where $L_{NF}(\xi) = (-\xi)^\ell / (1+ \xi)^\ell$, and $M_{L/R}(\xi)$ and their inverses are analytic and uniformly bounded on $\Omega$.
 \end{Lemma}
 
\begin{proof}
 We will just show the result only for $M_L(\xi)$ since a similar argument holds for $M_R(\xi)$.  Let
 \[ M_L(\xi ) = \left \{ 
\begin{array}{c c c}
L(\xi) L^{-1}_{NF}(\xi) & \mbox{for} & \xi \neq 0\\[2ex]
\lim_{\xi \rightarrow 0} L(\xi) L^{-1}_{NF}(\xi) & \mbox{for} & \xi =0 
\end{array}
\right. \]
Since both, $L_{NF}(\xi)$ and $L(\xi)$, are analytic, uniformly bounded, and invertible on $\Omega \bigcap \{ \xi \in \C: |\Re \xi |> \xi_m\} $, it follows that the same is true for $M_L(\xi)$. That $M_L(\xi)$ is analytic and bounded invertible near the origin follows from Riemann's removable singularity theorem and the following result,
\[ \lim_{\xi \rightarrow 0} L(\xi) L^{-1}_{NF}(\xi) =  \lim_{\xi \rightarrow 0} L(\xi) \frac{(1+ \xi)^\ell}{(-\xi)^\ell} = 1.\]

\end{proof}

We show next that the operator $\mathcal{M}_{L/R}: L^2_{\gamma}(\R^2) \longrightarrow L^2_{\gamma}(\R^2)$ defined by
\[
\begin{array}{c c c}
 L^2_{\gamma}(\R^2)& \longrightarrow &L^2_{\gamma}(\R^2)\\
u & \longmapsto & \mathcal{F}^{-1}\left ( M_{L/R}(|{\bf k}|^2) \hat{u} \right),
\end{array}
\]
is an isomorphism.

\begin{Lemma}\label{l:isomorphismMLR}
The operator $\mathcal{M}_{L/R}: L^2_{\gamma}(\R^2) \longrightarrow L^2_{\gamma}(\R^2)$ with Fourier symbol $M_{L/R}(| {\bf k }|^2)$ is an isomorphism for all $\gamma \in \R$.
\end{Lemma}
\begin{proof}
We first show that $\mathcal{M}_{L/R}$ are bounded from $L^2_\gamma(\R^2)$ to itself for values of $\gamma \in \N \bigcap \{ 0\}$ . Indeed this follows from Plancherel's Theorem and the results of Lemma \ref{l:decompositionconv}: Given $u \in L^2_\gamma(\R^2)$ we find that
\[ \| \mathcal{M}_{L/R}u ({\bf x}) \|_{L^2_\gamma} = \| M_{L/R}\hat{u} ({\bf k}) \|_{H^\gamma} \leq C(\|M\|_{C^\gamma} ) \| \hat{u} ({\bf k}) \|_{H^\gamma} =  C(\|M\|_{C^{\gamma}} )\| u ({\bf x}) \|_{L^2_\gamma}.\]
A similar argument shows that the their inverses, $\mathcal{M}_{L/R}^{-1}:L^2_{\gamma}(\R^2) \rightarrow L^2_{\gamma}(\R^2)$, are bounded.
We can extend this result to values of $\gamma \in \Z^-$ using duality. Then, because $H^\gamma(\R^2)$ is a complex interpolation space between $H^{\lfloor \gamma \rfloor}(\R^2)$ and $H^{\lfloor \gamma \rfloor +1}(\R^2)$, the result holds for all values of $\gamma \in \R$.
\end{proof}

From Lemmas \ref{l:decompositionconv} and \ref{l:isomorphismMLR} it follows that the convolution operator $\mathcal{L}:L^2_\gamma(\R^2) \rightarrow L^2_{\gamma}(\R^2)$ can be decomposed as
\[ \mathcal{L} u = \mathcal{M}_L \circ \mathcal{L}_{NF} u= \mathcal{L}_{NF} \circ \mathcal{M}_R u\]
with $\mathcal{L}_{NF} u = (Id-\Delta)^{-\ell} \Delta^\ell u$ and $\mathcal{M_{L/R} }: L^2_\gamma(\R^2) \rightarrow L^2_\gamma(\R^2)$ isomorphisms. Furthermore, both operators, $\mathcal{L}$ and $\mathcal{L}_{NF}$, share the same Fredholm properties. This proves Proposition \ref{p:fredholmconv}.

\begin{Remark*}
We note here that the integer $\ell$ that appears in the expression for $\mathcal{L}_{NF}$ is the same integer describing the form of the Fourier symbol $L(\xi)$ in Hypothesis \ref{h:multiplicity}. In particular, for  our problem we have that  $\ell=1$, so that for the remainder of the paper we will only consider this case.
\end{Remark*}

We are now in a position to define and describe the mapping properties of a related convolution operator, $\mathscr{M}$ which we will use to precondition equation \eqref{e:nonloc1}. In terms of   the invertible operator $\mathcal{M}_L$ defined above, $\mathscr{M}$ is given by
 \[\begin{array}{ c c c}
\mathscr{M}:H^s_\gamma(\R^2)& \longrightarrow &H^{s-2}_\gamma(\R^2),\\
  u & \longmapsto & (\Id-\Delta) \mathcal{M}_L^{-1}u.
 \end{array}
\]
It is clear from the definition that  $\mathscr{M} \ast c =c$ for all constants $c \in \R$.

In addition, we will also need to find mapping properties for, $\Id -\mathscr{M} $ and $\Id-J$. These operators will appear in the nonlinearities after preconditioning our equation with $\mathscr{M}$. We start with the next lemma which establishes the boundedness of the operator $\mathscr{M}$, which is straightforward to check.
\begin{Lemma}\label{l:convM}
For $\gamma \in \R$ and $s \in \Z$ the operator $\mathscr{M}: H^s_\gamma(\R^2) \rightarrow H^{s-2}_\gamma(\R^2)$ defined by 
\[ \mathscr{M}u = (\Id-\Delta)\mathcal{M}_L^{-1}u\]
is bounded.
\end{Lemma}

Next, we show that in the appropriate spaces the operator $\Id - \mathscr{M}$ behaves like the Laplacian operator.
\begin{Lemma}\label{l:conv1-M}
Let $\gamma \in \R$ and define  $\mathscr{M}$ as in Lemma \ref{l:convM}. Then, the convolution operator $$\Id- \mathscr{M}: M^{2,2}_{\gamma-2}(\R^2) \rightarrow L^2_\gamma(\R^2)$$ is a bounded operator. In particular, the operator is Fredholm with the same Fredholm properties as $\Delta:M^{2,2}_{\gamma-2}(\R^2) \rightarrow L^2_\gamma(\R^2)$.
\end{Lemma}
\begin{proof}
With the notation $\xi = | {\bf k}|^2$, we look at the Fourier Symbol of $\Id - \mathscr{M}$ and decompose it as
\[ \mathcal{F} ( \Id - \mathscr{M})(\xi)  = (1+ \xi^2) M^{-1}_L(\xi) \left (\frac{M_L(\xi)}{1 +\xi^2} -1 \right).\]
Since $\left (\frac{M_L(\xi)}{1 + \xi^2} -1 \right)$ satisfies Hypothesis \ref{h:analyticity} and Hypothesis \ref{h:multiplicity} with $\ell=1$, by Proposition \ref{p:fredholmconv} we can rewrite it as 
\[\left (\frac{M_L(\xi)}{1 + \xi^2} -1 \right) = F_L \frac{\xi^2}{1 + \xi^2 }, \]
where $F_L: H^s_\gamma(\R^2) \rightarrow H^s_\gamma(\R^2)$ is an isomorphism. The result of the lemma  follows from this decomposition.
\end{proof}

Lastly, we show next that the convolution operator $\Id - \mathcal{J}$ is also a Fredholm operator.
\begin{Lemma}\label{l:conv1-J}
The operator $\Id - \mathcal{J}: \mathcal{D} \rightarrow L^2_\gamma(\R^2)$, with domain 
$$\mathcal{D} =\{ u \in L^2_{\gamma-2}(\R^2) \mid (\Id- \Delta)^{-1} u \in M^{2,2}_{\gamma-2}(\R^2)\},$$
 has the same Fredholm properties as $\Delta (\Id-\Delta)^{-1} :\mathcal{D} \rightarrow L^2_\gamma(\R^2)$. Moreover, the operator $\Id -\mathcal{J} : M^{2,2}_{\gamma-2}(\R^2) \rightarrow H^2_{\gamma}(\R^2)$ is bounded. 
\end{Lemma}
\begin{proof}
Since by hypothesis $\mathcal{J}$ is radially symmetric and exponentially localized it follows that its Fourier symbol, $J({\bf k})$, is a radial function and that $J({\bf k}) \rightarrow 0$ as $| {\bf k} | \rightarrow \infty$. Moreover, since its average is 1, we must also have $J(0) =1$. This implies that $\mathcal{F}( \Id - \mathcal{J}) (| {\bf k}|)$ satisfies Hypothesis \ref{h:analyticity} and Hypothesis \ref{h:multiplicity} with $\ell=1$, and that its Taylor expansion near the origin is $\mathcal{F}( \Id - \mathcal{J}) (| {\bf k}|) \sim \rmO(|{\bf k}|^2)$. The results of this lemma now follow by Proposition \ref{p:fredholmconv}.
\end{proof}

We are now ready to show the existence of solutions for the intermediate approximation and the full solution.

\subsection{Intermediate Approximation}\label{s:nonlocalintermediate}
In this section we concentrate on the equation
\begin{equation}\label{e:nonlocalinter}
 0 =  \Delta \phi - \mathscr{M}\ast | \mathcal{J} \ast \nabla \phi|^2 + \eps \mathscr{M} \ast g(x,y), \quad (x,y) \in \R^2,
\end{equation}
which is obtained by preconditioning equation \eqref{e:nonloc1} with the operator $\mathscr{M}$. Here, the result $\mathscr{M} \ast \mathcal{L}= \Delta$ follows from Proposition \ref{p:fredholmconv} and Lemma \ref{l:convM}, since the operator $\mathcal{L}$ satisfies Hypothesis \ref{h:analyticity} and Hypothesis \ref{h:multiplicity} with $\ell=1$. 
Our goal is to show existence of solutions to \eqref{e:nonlocalinter} when $g$ is assumed to be an algebraically localized function. More precisely, with the definition for the cut off function, $\chi$, stated in the introduction, we prove the following proposition.

\begin{Proposition}\label{p:nonlocalintermediate}
Let $m \in \N$ and suppose $g$ is a localized function in the space $L^2_\sigma(\R^2)$, with weight strength $\sigma \in (m+1, m+2)$. Then, there exists an $\eps_0 >0$ and a $C^1$ map
\begin{equation*}
	\begin{matrix}
 \mathbb{\phi}&: (-\eps_0,\eps_0) & \longrightarrow & M^{2,2}_{\gamma-2}(\R^2) \times \R\\
& \eps &\longmapsto &(\phi, a_0)\\
 	\end{matrix}
\end{equation*}
 with $1<\gamma<2$, that allows us to construct an $\eps$- dependent family of solutions to equation \eqref{e:nonlocalinter}. Moreover, this solutions are of the form
\begin{equation*} \Phi(r,\theta;\eps) =   - \chi(r) \ln( 1 - a_0 \ln (r) )  + \eps \sum_{\substack{\alpha=-m \\ \alpha \neq 0}}^m a_{\alpha} \left( \frac{ \chi(r)}{r^{|\alpha|}}\right) \rme^{i \alpha \theta}  + \phi(r,\theta;\eps),
\end{equation*}
where,
\begin{itemize}

\item the constant $a_0 = \eps a_{0,1} + \rmO(\eps^2)$, with $a_{0,1} = -\frac{1}{2\pi} \int_{\R^2}    g({\bf x}) \;d{\bf x}$; 

\item for $\alpha \in [-m,m] \bigcap \Z \backslash\{0\} $ the constants 
 $a_\alpha = \frac{1}{2\pi \alpha} \int_0^\infty \int_0^{2\pi}    \mathscr{M} \ast g(r,\theta) r^{\alpha} \rme^{i\alpha \theta} d\theta \;r dr$; and
 \item the function $\phi < C r^{\gamma-1 } $ as $r \rightarrow \infty$.
\end{itemize}
\end{Proposition}

The proof is based on the following ansatz,
\begin{equation}\label{e:ansatznonlocal1}
 \phi(r,\theta;\eps) =  \psi_0(r) +\eps \sum_{\substack{ n\in \Z \\ n \neq0}} \psi_n(r) \rme^{\rmi \theta n}+ \eps \sum_{\substack{ \alpha=-m \\ \alpha \neq 0}}^m a_{\alpha}C_\alpha(r)  \rme^{i \alpha \theta} + \bar{\phi}(r,\theta;\eps)
\end{equation}
where the function $$\psi_0(r) = -\chi(r) \ln(1- a_0 \ln(r) ),\quad a_0  \in \R,$$ is motivated by the asymptotic analysis of Section \ref{s:eikonal}. In particular, notice that the function $  \ln(1- a_0 \ln(r) )$ solves the equation $\Delta \psi -(\partial_r \psi)^2 =0,$ and it admits the following series expansion 
\[ - \ln(1 -a_0\ln(r) ) =  a_0 \ln(r) + \frac{1}{2}(a_0 \ln(r) )^2 +\frac{1}{3}(a_0 \ln(r) )^3 + \cdots,\]
provided $|a_0\ln(r)| \ll1$. As a result we see that
\[ \Delta (\psi_0 - a_0 \chi \ln(r) ) - (\partial_r \psi_0)^2  + \Delta a_0 \chi \ln(r) =  \Delta a_0 \ln(r)  + \mbox{localized},\]
and it is this last equality which proves useful in the simplification of equation \eqref{e:nonlocalinter} and one of the reason why $\psi_0$ was picked as we have done.

 In addition, we let  $C_\alpha(r) = \frac{\chi}{r^{|\alpha|}}$ and  for the moment we assume that for $n \neq 0$ the pair $(\psi_n, a_n) \in M^{2,2}_{r, \sigma-2}(\R^2) \times \R$ solves the equations for the first order approximation to equation \eqref{e:nonlocalinter},
 \begin{eqnarray}\label{e:Delta_n0}
\partial_{rr} \psi_n + \frac{1}{r} \partial_r \psi_n - \frac{n^2}{r^2} \psi_n + a_n \Delta_n C_n  + \tilde{g}_n(r)  =0& \mbox{for}\quad 0<|n|\leq m\\ \label{e:Delta_n00}
\partial_{rr} \psi_n + \frac{1}{r} \partial_r \psi_n - \frac{n^2}{r^2} \psi_n   +  \tilde{g}_n(r)  =0& \mbox{for}\quad m < |n|.
\end{eqnarray}
Here, the inhomogeneous terms $\tilde{g}_n$, for $n \in \Z$, are just the polar modes of $\tilde{g} =\mathscr{M} \ast g$ obtained using Lemma \ref{l:polar}.

We now wish to look at the operator $F: M^{2,2}_{\gamma-2}(\R^2)\times \R\times \R \rightarrow L^2_\gamma(\R^2)$, obtained by inserting the ansatz \eqref{e:ansatznonlocal1} into equation \eqref{e:nonlocalinter}. A straight forward calculation using the results from Lemma \ref{l:polar} and the choice of $\psi_0$ shows that 
\begin{align*}
F(\bar{\phi}, a_0; \eps)&= \Delta \bar{\phi} + a_0 \Delta \chi \ln(r)+ \eps  \tilde{g}_0 + P(\bar{\phi}),\\
P(\bar{\phi}) &=  - \mathscr{M} \ast | \mathcal{J} \ast \nabla \phi|^2 + | \nabla \psi_0|^2 + ( \Delta \psi_0 - a_0 \Delta \chi \ln(r) - |\nabla \psi_0|^2).
\end{align*}
The dependence on $\bar{\phi}$ in the definition of the nonlinearity $P$ comes through the function $\phi$, which is just the ansatz \eqref{e:ansatznonlocal1}.
 
To show the results of Proposition \ref{p:nonlocalintermediate} we apply the implicit function theorem to find the zeros of $F$. It is easy to check that $F$ depends smoothly on $\eps$, so that we are left with showing three things:
\begin{enumerate}
\item \label{invertible} The equations \eqref{e:Delta_n0} and \eqref{e:Delta_n00} are invertible.
\item The operator $F: M^{2,2}_{\gamma-2}(\R^2)\times \R\times \R \rightarrow L^2_\gamma(\R^2)$ is well defined for $1<\gamma<2$.
\item \label{Frechet} The derivative $D_{\bar{\phi},a_0} F \mid_{\eps =0}  = \mathscr{L}$ defined as $\mathscr{L} :M^{2,2}_{\gamma-2}(\R^2) \times \R \rightarrow L^2_\gamma(\R^2)$,
\[ \mathscr{L}(\phi, a) = \Delta \phi + a \Delta \chi \ln(r)\]
is invertible.
\end{enumerate}
  
  We start with item~(\ref{invertible}). From Lemma \ref{l:Delta_nFredholm} we know that in the setting of radial Kondratiev spaces  and for $n \in \N \bigcup \{0\}$ the operators, 
\begin{equation}\label{e:Delta_ndefinition}
\begin{array}{c c c}
\Delta_n : M^{2,2}_{r,\sigma -2}(\R^2) & \longrightarrow &L^2_{r,\sigma}(\R^2),\\
\phi & \longmapsto &\partial_{rr} \phi  + \frac{1}{r} \partial_r \phi - \frac{n^2}{r^2}\phi 
\end{array}
\end{equation}
 are invertible for values of $\sigma \in (1-n,1 + n)$ and are Fredholm index $ i= -1$ for values of $\sigma > n+1$. It is for this reason, and our assumption that $m+1<\sigma<m+2$, that we have added correction terms, $C_\alpha(r)$, only for values of $0< |n| \leq m$. In particular, we show in the next lemma that the functions $\Delta_n C_n$ span the cokernel of the operator \eqref{e:Delta_ndefinition}, so that by an application of Lyapunov Schmidt reduction it follows that the equations, \eqref{e:Delta_n0} and \eqref{e:Delta_n00}, are solvable and that the value of the constants, $a_n$, is given by
\[ a_n = \frac{1}{n}\int_0^\infty \tilde{g}_n(r) r^n \;rdr, \quad  \mbox{for} \quad  0<|n|\leq m.\]
Moreover, since the functions $\psi_n$ are in $M^{2,2}_{r,\sigma-2}(\R^2)$ for values of $\sigma \in (m+1,m+2)$, the results from Lemma \ref{l:decay} imply that these functions satisfy $ | \psi_n| < C r^{-2m}$ for large values of $r$.

\begin{Lemma}\label{l:invertibleDeltan}
Let $ n \in \N$, take $\gamma> 1 + n$, and define  $\Delta_n = \partial_{rr} + \frac{1}{r} \partial_r + \frac{n}{r^2}$. Then, the operator $\mathcal{L}_n : M^{2,2}_{r,\gamma-2}(\R^2) \times \R \rightarrow L^2_{r,\gamma}(\R^2)$ defined as
\[ \mathcal{L}_n(\phi,a) = \Delta_n \phi + a \Delta_n \left( \frac{\chi}{r^{|n|}} \right)
  \]
is invertible. 
\end{Lemma}
\begin{proof}
This is a consequence of Lemma \ref{l:Delta_nFredholm} and the fact that the function $\Delta_n \left( \displaystyle \frac{\chi}{r^{|n|}}\right)$ spans the cokernel of $\Delta_n: M^{2,2}_{r,\gamma-2}(\R^2) \longrightarrow L^2_{r, \gamma}(\R^2)$, for $n \in \N$. Indeed, a short calculation shows that 
\[  \int_0^{\infty} \Delta_n \left( \frac{\chi}{r^{|n|}} \right) r^n \; rdr= -n.\]

\end{proof}

  Next, we show that the operator $F$ is well defined.
  \begin{Lemma}\label{l:nonlocalnonlinearinter}
  Let $\gamma \in (1,2)$, then the operator $P: M^{2,2}_{\gamma-2}(\R^2) \rightarrow L^2_\gamma(\R^2)$ given by
\[P(\bar{\phi}) = - \mathscr{M} \ast | \mathcal{J} \ast \nabla \phi|^2 + | \nabla \psi_0|^2 + ( \Delta \psi_0 - a_0 \Delta \chi \ln(r) - |\nabla \psi_0|^2),\] 
  and with $\phi(r,\theta;\eps)$ as in \eqref{e:ansatznonlocal1},   is well defined.
\end{Lemma}
  
\begin{proof}

Since $\psi_0(r) = - \chi \ln( 1-a_0 \ln(r))$, we know that the expression $( \Delta \psi_0 - a_0 \Delta \chi \ln(r) - |\nabla \psi_0|^2)$ is localized.

To show that the remaining terms are in $L^2_\gamma(\R^2)$ we let $ \Phi(r,\theta;\eps) = \phi(r,\theta;\eps) - \psi_0$, and write
\[ \mathscr{M} \ast | \mathcal{J} \ast \nabla \phi|^2 = \mathscr{M} \ast \left[ | \mathcal{J} \ast \nabla \Phi|^2 + 2 (\mathcal{J} \ast \nabla \Phi) \cdot (\mathcal{J} \ast \nabla \psi_0) +| \mathcal{J} \ast \nabla \psi_o|^2 \right].\]

From Hypothesis \ref{h:nonlinearity} it is not hard to see that the operator $\mathcal{J}$ satisfies,
\[ \mathcal{J}: M^{s,2}_\gamma(\R^2) \longrightarrow \{ u \in M^{s,2}_\gamma(\R^2) \mid D^s u \in H^2_{\gamma+s}(\R^2)\} ,\quad \forall \gamma \in \R.\]
We can also check that the term $\Phi$ is in the space $M^{1,2}_{\gamma-1}(\R^2)$, which together with the mapping properties of $\mathcal{J}$ and Lemma \ref{l:decay}, implies that the expression  $(\mathcal{J} \ast \nabla \Phi)$ is bounded and decays  as $|x|^{-\gamma}$. As a result the function $| \mathcal{J} \ast \nabla \Phi|^2 \in H^2_\gamma(\R^2) $ and we may conclude from the mapping properties of $\mathscr{M}$ stated in Lemma \ref{l:convM} that the term  $\mathscr{M} \ast  |\mathcal{J} \ast \nabla \Phi|^2$ is well defined.

Similarly, since $\nabla \psi_0 \in M^{2,2}_\tau(\R^2)$ with $-1< \tau<0$, it follows from Lemma \ref{l:decay} and the mapping properties of $\mathcal{J}$ that $|\mathcal{J} \ast \nabla \psi_0 | $ decays as $ |x|^{-1}$. Consequently,  the function $(\mathcal{J}\ast \nabla \psi_0) \cdot (\mathcal{J} \ast \nabla \Phi)$ is in the space $H^2_\gamma(\R^2)$ and the product $\mathscr{M} \ast(\mathcal{J}\ast \nabla \psi_0) \cdot (\mathcal{J} \ast \nabla \Phi)$ is well defined.

Lastly, we look at the expression $\mathscr{M} \ast | \mathcal{J} \ast \nabla \psi_0|^2 - |\nabla \psi_0|^2$, and rewrite it as follows
\begin{align*}
 \mathscr{M} \ast | \mathcal{J} \ast \nabla \psi_0|^2 - |\nabla \psi_0|^2  = & \mathscr{M} \ast \left[ | (\mathcal{J}- \Id) \ast \nabla \psi_0|^2 + 2(\mathcal{J}-\Id) \ast \nabla \psi_0 \cdot \psi_0 \right]  \\
& + (\mathscr{M}- \Id) \ast |\nabla \psi_0|^2
\end{align*}
Using Lemma \ref{l:conv1-J} we see that the term $(\mathcal{J}-\Id) \ast \nabla \psi_0$ is in the space $H^2_{\gamma}(\R^2)$, which being a Banach algebra implies that the same term squared is in $H^2_{\gamma}(\R^2)$. Then, the mapping properties of $\mathscr{M}$ and the boundedness of $\nabla\psi_0$ show that the first term of the above expression is well defined. To show that the last term is also in the space $L^2_\gamma(\R^2)$ we use Lemma \ref{l:multiplication} to check that  $|\nabla \psi_0|^2 $ is in $M^{2,2}_\tau(\R^2)$ and then appeal to the results of Lemma \ref{l:conv1-M}.

\end{proof}
To complete the proof of Proposition \ref{p:nonlocalintermediate} we need to show item~(\ref{Frechet}). Since the nonlinearities are in the space $L^2_\gamma(\R^2)$ with $\gamma \in (1,2)$ the results from Theorem \ref{McOwen} show that the operator $\Delta : M^{2,2}_{\gamma-2}(\R^2) \rightarrow L^2_\gamma(\R^2)$ is injective with cokernel spanned by $1$.  A short calculation shows that $\int_{\R^2} \Delta \chi \ln(r) \;dx = 2\pi,$ from which it follows that the operator $\mathscr{L} :M^{2,2}_{\gamma-2}(\R^2) \times \R \rightarrow L^2_\gamma(\R^2)$,
\[ \mathscr{L}(\phi, a) = \Delta \phi + a \Delta \chi \ln(r)\]
is invertible. If we now consider regular expansions for $\bar{\phi}$ and $a_0$, and apply Lyapounov Schmidt reduction to the equation
\begin{equation*}
\Delta \bar{\phi} + a_0 \Delta \chi \ln(r)+ \eps  \tilde{g}_0 + P(\bar{\phi})=0,
\end{equation*}
 the result is that $a_0 = \eps a_{0,1} + \rmO(\eps^2)$ with 
\[a_{0,1} =- \frac{1}{2\pi} \int_{\R^2} \tilde{g}_0= -\frac{1}{2\pi} \int_{\R^2} g.\]
 In addition, using Lemma \ref{l:decay} we see that the function $\bar{\phi} < C r^{\gamma-1}$ for large values of $r$. This completes the proof of Proposition \ref{p:nonlocalintermediate}.

\subsection{ Full solution for fixed $\omega= \lambda^2$}\label{s:nonlocalfull}
The following equation is obtained by preconditioning \eqref{e:nonloc1} with the operator $\mathscr{M}$ and then assuming that we  have an ansatz $\phi(x,y,t, \eps, \lambda) = \tilde{\phi}(x,y, \eps) - \lambda^2 t$, where $\lambda$ is another parameter. Dropping the tildes from our notation:
\begin{equation}\label{e:eiknonlocal2}
 -\lambda^2=  \Delta \phi - \mathscr{M} \ast | \mathcal{J} \ast \nabla \phi|^2  + \eps \mathscr{M} \ast g(x,y) \quad (x,y) \in R^2.
\end{equation}
We are again interested in solutions that bifurcate from the steady state $\phi=0$ when $\eps \neq 0$ and $\lambda >0$. We point out that in our original problem the value of $\lambda^2$ depends on $\eps$, since it represents the frequency, $\omega$, of the target waves that emerge from the introduction of a perturbation. In the next subsection  we will find the value of $\omega =\lambda^2$ by matching the solutions from Section \ref{s:nonlocalintermediate} with the solutions from Proposition \ref{p:nonlocalfull} stated next.
\begin{Proposition}\label{p:nonlocalfull}
Suppose $g$ is in the space $ L^2_\sigma(\R^2)$ with $\sigma >0$. Then, there exists numbers $\eps_0, \lambda_0 >0$ and a $C^1$ map
\begin{equation*}
	\begin{matrix}
	\Gamma:& (-\eps_0,\eps_0) \times [0,\lambda_0)& \longrightarrow& \mathcal{D} \subset M^{1,2}_{\gamma-1}(\R^2)\\
	& (\eps,\lambda) &\mapsto & \phi
	\end{matrix}
\end{equation*}
where $\gamma>0$ and $ \mathcal{D} = \{\phi \in M^{1,2}_{\gamma-1}(\R^2) \mid  \nabla\phi \in H^1_\gamma(\R^2)  \}$, that allows us to construct an $(\eps,\lambda)$-dependent family of solutions to equation \eqref{e:eiknonlocal2}. Moreover, these solutions are of the form,
\begin{equation*}
\Phi(r,\theta;\eps,\lambda) =  - \chi(\lambda r) \ln( K_0( \lambda r)+  \phi(r,\theta;\eps,\lambda),
\end{equation*}
where
\begin{itemize}
\item  the function $K_0(z)$ is the zero-th order modified Bessel's function of the second kind, and
\item the function $\phi < C r^{-\gamma } $ as $r\rightarrow \infty$.
\end{itemize}
\end{Proposition}

To show the results of the above proposition we again consider an appropriate ansatz motivated by the results from Section \ref{s:eikonal}
\begin{equation}\label{e:ansatznonlocal2}
\phi(r,\theta) = -\chi(\lambda r) \ln( K_0(\lambda r) ) + \bar{\phi}(r,\theta).
\end{equation}

For ease of notation we define $\psi_0(r) =-\chi( \lambda r) \ln( K_0(\lambda r) )$ and remark that the function $-\ln( K_0(\lambda r ) )$ is a solution to the equation $-\lambda^2 = \Delta \psi_0 - |\nabla \psi_0|^2$.  As a consequence, when inserting the above ansatz into equation \eqref{e:eiknonlocal2} we find that $\bar{\phi}$ must satisfy
\begin{equation}\label{e:nonlocal2}
0 = \Delta \bar{\phi} - 2\lambda \partial_r \bar{\phi} + \eps \mathscr{M} \ast g(r,\theta) + P(\bar{\phi}),\end{equation}
where the nonlinear term $P(\bar{\phi})$ is given by
\[ P(\bar{\phi}) = -\mathscr{M} \ast| \mathcal{J} \ast \nabla \phi|^2 + |\nabla \psi_0|^2 + 2\lambda \partial_r \bar{\phi} +( \lambda^2 + \Delta \psi_0 - |\nabla \psi_0|^2),\]
and again $\phi$ is as in ansatz \eqref{e:ansatznonlocal2}.

We have chosen to use polar coordinates to highlight the form of the linearization and to take advantage of the polar decomposition of Kondratiev spaces. As in the previous section, we make use of the right hand side of equation \eqref{e:nonlocal2} as an operator $F: \mathcal{D} \times \R\times \R \rightarrow \mathcal{D}$,
\[ F(\bar{\phi}; \eps, \lambda) = \Id - \mathscr{L}_\lambda^{-1}\left ( \eps \mathscr{M} \ast g(r,\theta)  +P(\bar{\phi})  \right).\]
where $ \mathcal{D} = \{\phi \in M^{1,2}_{\gamma-1}(\R^2) \mid  \nabla\phi \in H^1_\gamma(\R^2)  \}$, and $\mathcal{L}_\lambda $ is defined as  $\mathscr{L}_\lambda: \mathcal{D} \subset M^{1,2}_{\gamma-1}(\R^2) \longrightarrow L^2_\gamma(\R^2)$
  \[ \mathscr{L}_\lambda \phi = \Delta \phi - 2 \lambda \partial_r \phi, \quad \lambda \geq0.
  \]
  
   The results of the proposition then follow by finding the zeroes of $F$ via the implicit function theorem, keeping in mind that the operator is only defined for values of $\lambda \geq0$. We first notice that $F$ depends smoothly on $\eps$  for all values of $\eps \in \R$, that $F(0;0,0)=0$, and that the Fr\'echet derivative, $D_{\phi}F\mid_{(\eps,\lambda) = (0,0)} = I$, is invertible. Thus, we are left with showing the following two points:
   \begin{enumerate}
 \item \label{item:defined} The operator  $F: \mathcal{D} \times \R\times \R \rightarrow \mathcal{D}$ is well defined, and
 \item \label{item:smooth} it is $C^1$ with respect to $\lambda \in [0, \lambda_0)$ for some $\lambda_0>0$.
  \end{enumerate}
 
  The results of item \ref{item:defined}.) are true if we can show that the nonlinearities are well defined in $L^2_\gamma(\R^2)$, and if the linear operator $\mathcal{L}_\lambda$ is invertible. The following lemma shows that the first assertion is true.

\begin{Lemma}\label{l:nonlocalnonlinearfull}
Let $\gamma>0$, then the operator $P:\mathcal{D} \rightarrow L^2_\gamma(\R^2)$, defined by
\[ P(\bar{\phi}) = -\mathscr{M} \ast| \mathcal{J} \ast \nabla \phi|^2 + |\nabla \psi_0|^2 + 2\lambda \partial_r \bar{\phi} +( \lambda^2 + \Delta \psi_0 - |\nabla \psi_0|^2)\]
and with $\phi(r,\theta;\eps)$ as in equation \eqref{e:ansatznonlocal2}, is bounded.
\end{Lemma}
  \begin{proof}
 Because  $\psi_0 = -\chi \ln( K_0(\lambda r))$, the expression in the parenthesis is localized. To show that the remaining terms are also well defined in the space $L^2_\gamma(\R^2)$ we write
  \[ \mathscr{M} \ast | \mathcal{J} \ast \nabla \phi|^2 = - \mathscr{M} \ast \left[ | \mathcal{J} \ast \nabla \bar{\phi}|^2 - 2 (\mathcal{J} \ast \nabla \bar{\phi}) \cdot ( \mathcal{J} \ast \nabla \psi_0) + | \mathcal{J} \ast \nabla \psi_0|^2 \right]. \]
  Since $\bar{\phi}$ is in the space $\mathcal{D}$ it follows from Hypothesis \ref{h:nonlinearity} that the term $(\mathcal{J}\ast \nabla \bar{\phi})$ is in $H^3_\gamma(\R^2)$, and it is then clear that the expression, $ \mathscr{M} \ast | \mathcal{J} \ast \nabla \bar{\phi}|^2$ is in $L^2_\gamma(\R^2)$. 
  
  Next, we look at the difference
  \begin{align*}
   \lambda \partial_r \bar{\phi} -\mathscr{M}\ast \left[ (\mathcal{J} \ast \nabla \bar{\phi}) \cdot ( \mathcal{J} \ast \nabla \psi_0) \right] = &  \lambda \partial_r \bar{\phi} -\mathscr{M}\ast \left[ (\mathcal{J} \ast \partial_r \bar{\phi})  ( \mathcal{J} \ast \partial_r \psi_0)\right]\\
 =& \lambda \partial_r \bar{\phi} -\mathscr{M}\ast \left[ (\mathcal{J} \ast  \partial_r \bar{\phi} )  ( \mathcal{J} \ast (\partial_r \psi_0- \lambda + \lambda)) \right]  \\
 &  + \lambda \mathcal{J} \ast \partial_r \bar{\phi}- \lambda \mathcal{J} \ast \partial_r \bar{\phi} \\
  =& \lambda (\mathcal{J} - \Id) \ast \partial_r \bar{\phi} - \lambda \mathscr{M} \ast ( \mathcal{J} \ast \partial_r \bar{\phi}) \\
  &  - \mathscr{M} \ast \left[ ( \mathcal{J} \ast ( \partial_r \psi_0 - \lambda) (\mathcal{J} \ast\partial_r \bar{\phi} ) ) \right],
  \end{align*}
 where we used the fact that $\psi_0(r)$ is a radial function to do the simplifications. Using Lemmas \ref{l:convM} and \ref{l:conv1-J}, and because $\partial_r \bar{\phi}$
  is in the space $H^2_\gamma(\R^2)$, one is able to check that the first two terms in the last equality are well defined functions in $L^2_\gamma(\R^2)$. The last term is localized since $\partial_r \psi_0 = \frac{\lambda K'_0(\lambda r)}{ K_0(\lambda r)} $ approaches $\lambda$ as $r $ goes to infinity, see Table \ref{tab:Bessel} in Section \ref{s:eikonal}.
  
 Finally, we consider the difference
 \begin{align*}
 \mathscr{M}\ast | \mathcal{J} \ast \nabla \psi_0|^2 - |\nabla \psi_0|^2 =& \mathscr{M}\ast |\mathcal{J} \ast (\partial_r \psi_0 - \lambda + \lambda )|^2 - (\partial_r\psi_0)^2\\
 =&\mathscr{M}\ast \left [ |\mathcal{J} \ast (\partial_r \psi_0 - \lambda)|^2 +2 \lambda \mathcal{J} \ast( \partial_r \psi_0 - \lambda ) \right] + \lambda^2 - (\partial_r\psi_0)^2,
 \end{align*}
where in the he last equality  we used the fact that $\mathscr{M}\ast \lambda = \lambda$.  Because the function $\partial_r \psi_0$ approaches $\lambda$ in the far field, It is now easy to see that the above expression is localized.
\end{proof}

Our next task is to show that the linear operator $\mathscr{L}_\lambda :\mathcal{D} \longrightarrow L^2_\gamma(\R^2)$ is invertible. To that end we use Lemma \ref{l:polar} and decompose $\phi$ into its polar modes,
\[ \mathscr{L}_\lambda \phi = \sum_{n \in \Z} \left(\partial_{rr} \phi_n + \frac{1}{r} \partial_r \phi_n - \frac{n^2}{r^2} \phi_n - 2\lambda \partial_r \phi_n \right) \rme^{\rmi n \theta}. \]
It follows that $\mathscr{L}_\lambda$ is invertible if for each $n \in \Z$ the following operators are also invertible:
\[\begin{array}{c c c}
 \mathscr{L}_{\lambda,n} : \mathcal{D}_n &\longrightarrow &L^2_{r,\gamma}(\R^2)\\
\phi &\longmapsto& \partial_{rr} \phi + \frac{1}{r} \partial_r \phi - \frac{n^2}{r^2} \phi - 2\lambda \partial_r \phi.
\end{array}\]
Here we used the notation introduced in Section \ref{s:weightedspaces}, and defined $\mathcal{D}_n = \{ u \in m^n_{\gamma-1} \mid \partial_ru \in L^2_{r,\gamma}(\R^2), \partial_{rr} u \in L^2_{r,\gamma}(\R^2)\}$, so that $\mathcal{D} = \bigoplus \mathcal{D}_n$. Since the multiplication operator $\frac{n^2}{r^2}: \mathcal{D}_n \rightarrow L^2_{r,\gamma}(\R^2)$ is compact, the invertibility of  $\mathscr{L}_{\lambda,n}$ follows from Lemma \ref{l:fredholmLambda}, provided $\gamma>0$. We summarize this result in the following lemma.
\begin{Lemma}\label{l:invertibleLambda}
Given $\gamma>0$, the operator $\mathscr{L}_{\lambda} : \mathcal{D} \longrightarrow L^2_{\gamma}(\R^2),$
\[ \mathscr{L}_{\lambda} \phi =\Delta \phi -2  \lambda  \partial_r \phi\]
is invertible
\end{Lemma}

Finally, to show item 2) we prove that $\mathscr{L}_\lambda^{-1} : L^2_\gamma(\R^2) \rightarrow \mathcal{D}$ depends continuously on the parameter $\lambda$.

\begin{Lemma}\label{l:continuousLambda}
Given $\gamma>0$, the operator $\mathscr{L}_{\lambda} : \mathcal{D} \longrightarrow L^2_{\gamma}(\R^2),$
\[ \mathscr{L}_\lambda  \phi = \Delta \phi - 2 \lambda \partial_r \phi,\]
and its inverse, are both $C^1$ in $\lambda$ for values of $\lambda \geq0$. 
\end{Lemma}
\begin{proof}
That the operator $\mathscr{L}_\lambda$ is continuous with respect to the parameter $\lambda$ is straightforward. To show the same result for its inverse we will use the notation $\mathscr{L}_\lambda =\mathscr{L}(\lambda)$ to highlight the dependence of the operator on the parameter $\lambda$. At the same time, for  $f \in L^2_\gamma(\R^2)$ and for general $\lambda \geq 0$, we denote by $\phi(\lambda)$ the solution to $\mathscr{L}_\lambda\phi =f$ and we look at this next equality
\[ \phi(\lambda + h \lambda) - \phi(\lambda)  = - \mathscr{L}(\lambda)^{-1} \left[ \mathscr{L}(\lambda + h \lambda) - \mathscr{L}(\lambda) \right] \phi(\lambda + h \lambda).\]
Because the operator $\left[ \mathscr{L}(\lambda + h \lambda) - \mathscr{L}(\lambda) \right]$ is bounded from $\mathcal{D}$ to $L^2_\gamma(\R^2)$ it follows that $\mathscr{L}^{-1}(\lambda)$ is continuous with respect $\lambda$. A short calculation shows that the derivative of $\mathscr{L}^{-1}(\lambda)$ with respect to $\lambda$ is an operator from $L^2_\gamma(\R^2)$ to $\mathcal{D}$ of the form $\lambda \mathscr{L}^{-1}(\lambda) \partial_r \mathscr{L}^{-1}(\lambda)$, and a similar analysis justifies that this operator is indeed continuous with respect to $\lambda$.
\end{proof}
Notice that because $\phi \in \mathcal{D} \subset M^{2,2}_{\gamma-1}(\R^2)$,  Lemma \ref{l:decay} implies that $|\phi|< C r^{-\gamma}$ for large values of $r$. This completes the proof of Proposition \ref{p:nonlocalfull}

\subsection{Matching and proof of the main theorem}\label{s:nonlocalmatching}
We first find a relation between the frequency, $\omega$, and the parameter $\eps$, and then proceed to show the results from Theorem \ref{t:main}.

Given $g \in L^2_\sigma(\R^2)$ , with values of $\sigma>1$, we know from Proposition \ref{p:nonlocalintermediate} that the intermediate solution is of the form 
$$\Phi(r, \theta;\eps) = -\chi(r) \ln(1- a_0 \ln(r))+ \tilde{\phi}(r,\theta; \eps),$$ 
with  $\tilde{\phi} \leq  r^{-\delta}$, for large $r$ and for values of $\delta \in (0,1)$. At the same time the results from Section \ref{s:nonlocalfull} show that solutions to \eqref{e:eiknonlocal2} are given by
\[\Psi(r, \theta;\eps,\lambda) = -\chi(\lambda r) \ln( K_0(\lambda r)) + \tilde{\psi}(r,\theta;\eps,\lambda)\]
where again $\tilde{\psi} \leq r^{-\delta}$, for values of $\delta >0 $ and large $r$.

Since the intermediate approximation is only valid so long as $ | \ln(r) | \ll-1/a_0$,  and because we are assuming $\lambda$ is small beyond all orders, we may pick the scaling $ r= \eta \eps / \lambda $, with $\eta \sim \rmO(1)$, to do the matching. Notice that in this scaling the function $K_0(\lambda r) \sim -\ln(\lambda r/ 2) - \gamma_e$, where  $\gamma_e$ is the Euler constant. We can then match the wavenumbers, $\nabla \Phi$ and $\nabla \Psi$, to obtain   $ \lambda =  2 \rme^{ -\gamma_e}  \exp( -1/a_0) $. This in turn gives us an expression for $\omega$:
\begin{equation}\label{e:omega}
 \omega = 4\rme^{-2\gamma_e} \exp( -2/a_0)\sim 4 \, C(\eps) \rme^{-2\gamma_e} \exp\left( \frac{ 4 \pi }{ \eps \int_{\R^2} g} \right), \end{equation}
where we use the approximation $a_0 =  a_{0,1} \eps + a_{0,2} \eps^2 + \rmO(\eps^3)$, with $a_{0,1} = -\frac{1}{2\pi} \int_{\R^2} g$ so that
\[
C(\eps) = \exp\left(\frac{2}{\eps a_{0,1}} - \frac{2}{a_0}\right) = \exp\left( \frac{2 C'}{a_{0,1}^2}\right) + \rmO(\eps)
\].

\begin{Remark*}\label{r:omega}Notice that:

\begin{enumerate}
\item The above expression shows that if  $\eps \int_{\R^2} g(x) dx <0$ the frequency $\omega$ is smaller than $\eps$ beyond all orders, which is consistent with our assumptions. In addition,  $\omega(\eps)$ depends smoothly on $\eps$ on the interval $\eps \in (0,\infty)$, and by defining $\omega(0) = \partial_\eps \omega( 0) =0$ we may also conclude that $\omega(\eps)$ is continuously differentiable with respect to $\eps$ on the closed interval $[0,\infty)$. 
\item In Section \ref{s:eikonal} we derived the following approximation for the frequency: $$\omega \sim \frac{4 \rme^{-2 \gamma_e} }{r_c^2}\exp\left(\frac{4 \pi}{\eps \int_{\R^2} g}\right).$$  The constant $r_c$ that appears in this expression is therefore accounted for by the constant $C(\eps)$ in \eqref{e:omega}.
\end{enumerate}
\end{Remark*}

We can now prove our primary result, Theorem \ref{t:main}, for the nonlocal eikonal equation 
$$
 \phi_t = \mathcal{L} \ast \phi- |\mathcal{J} \ast \nabla \phi|^2 + \eps g(x,y) \quad (x,y) \in \R^2.
$$

\setcounter{Theorem}{0}

\begin{Theorem}
Suppose that the kernels $\mathcal{L}$ and $\mathcal{J}$ satisfy Hypotheses~\ref{h:analyticity},~\ref{h:multiplicity} with $\ell =1$, and~\ref{h:nonlinearity}. Additionally, suppose $g$ is in the space $  L^2_\sigma(\R^2)$ with $\sigma>1$ and let $M= \frac{1}{2\pi} \int_{\R^2} g <0$. Then, there exists a number $\eps_0>0$ and a $C^1$ map
\begin{equation*}
	\begin{matrix}
	\Gamma:& [0,\eps_0) &\longrightarrow & \mathcal{D} \subset M^{2,1}_{\gamma-1}(\R^2)\\
	& \eps& \longmapsto & \psi\\
	\end{matrix}
\end{equation*}

where $\gamma>0$ and $ \mathcal{D} = \{\phi \in M^{1,2}_{\gamma-1}(\R^2) \mid  \nabla\phi \in H^1_\gamma(\R^2)  \}$, that allows us to construct an $\eps$-dependent family of target pattern solutions to \eqref{e:nonloc1}. Moreover, these solutions have the form
\[\Phi(r,\theta, t;\eps) = -\chi(\lambda(\eps) r) \ln( K_0(\lambda(\eps) r) ) + \psi(r, \theta ;\eps)- \lambda^2(\eps) t, \quad \lambda(\eps) >0.\]
In particular, as $r \rightarrow \infty$,
\begin{itemize}
\item $ \psi(r;\eps) < C r^{-\delta}$,  for $C \in \R$ and $\delta \in (0,1)$;  and
\item $ \lambda(\eps)^2 \sim 4\,C(\eps) \,\rme^{-2\gamma_e} \exp\left( \frac{ 2 }{ \eps M}\right)$,  where  $C(\eps) $ represents a constant that depends on $\eps$,  and $\gamma_e$ is the Euler constant.
\end{itemize}
In addition, these target pattern solutions have the following asymptotic expansion for their wavenumber
 \[ k( \eps) \sim  \exp\left (\frac{1}{\eps M} \right ) + \rmO\left( \frac{1}{r^{\delta+1}}\right)\quad \mbox{as} \quad r \rightarrow \infty.\]
\end{Theorem}
\begin{proof}
Considering solutions $\Phi(x,t) = \phi(x) - \lambda^2 t$ we arrive equation \eqref{e:eiknonlocal2}. Then, letting 
\[\phi(r,\theta, t;\eps) = -\chi(\lambda(\eps) r) \ln( K_0(\lambda(\eps) r) ) + \psi(r, \theta ;\eps) - \lambda^2(\eps) t, \quad \lambda(\eps) >0,\]
 the above analysis together with the matched asymptotics shows the smooth dependence of $\lambda$ on the parameter $\eps$, for $\eps \in (0,\infty)$. We can further define $\lambda$ so that it is continuously differentiable with respect to $\eps$ as $\eps \rightarrow 0$. The same arguments as in Proposition \ref{p:nonlocalfull} can be adapted, but now the nonlinear operator $F$ specified in that proof depends only on $\eps$. Therefore, there exists an $\eps_0>0$ so that the above ansatz indeed represents an $\eps$-dependent family of solutions to \eqref{e:eiknonlocal2} for $\eps \in [0,\eps_0)$.  The theorem follows directly.
 \end{proof}

\section{Discussion} \label{s:discuss}

In this paper we derive a model nonlinear integro-differential equation  to describe the phase evolution of an array of oscillators with nonlocal diffusive coupling ( see equation ~\eqref{e:nonloc1} and Appendix~\ref{sec:hierarchy}). We prove the existence of target wave solutions in such systems when a pacemaker
 is introduced, which we model as a localized perturbation of the natural frequency of the oscillators. This work builds on earlier results by Doelman et al on phase dynamics for modulated wave trains \cite{doelman2005dynamics}, and on work by Koll\'ar and Scheel on radial patterns generated by inhomogeneities \cite{kollar2007coherent}. In contrast to the systems considered in the earlier papers, our model equation accounts for nonlocal interaction and nonradially symmetric perturbations, so our results are applicable to a wider class of systems. However, in other aspects, the earlier results are stronger and they motivate natural directions for future work. In particular:
\begin{enumerate}
\item Starting from a reaction-diffusion system, Doelman et al rigorously prove the validity of the (local) viscous eikonal equation as a reduced, modulational equation \cite{doelman2005dynamics}. Our nonlocal model~\eqref{e:nonloc1} generalizes this equation but its derivation in Appendix~\ref{sec:hierarchy} is purely formal. A natural question would be the rigorous validation of our model as a reduction of a nonlocal reaction-diffusion equation, e.g a neural field model.
\item On the other hand, Koll\'ar and Scheel showed the existence of target wave solutions for the ``original" reaction-diffusion system~\eqref{e:reaction-diffusion}, in a situation where the inhomogeneity $g = g(r)$ is a radial function. This motivates proving the existence of target patterns in the ``original" neural field or continuum coupled model  without appealing to a reduced model of phase dynamics.

\item Koll\'ar and Scheel also considered the case of perturbations with positive mass, i.e. the ``mass" $\eps \int g > 0$, and showed the existence of weak sinks and contact defects. The latter can be characterized as waves which propagate towards the core, but with wavenumber that converges to zero at infinity. We conjecture that a similar result should hold in our model~\eqref{e:nonloc1} for nonradial inhomogeneities and radial convolution kernels (where the origin is a distinguished point). We believe that this case is accessible to the methods presented in this paper.

\item The case of a mass zero perturbation was considered  previously in \cite{simon1976bound}. Here, the author shows  the existence of a weakly bound ground state with an energy $-\omega  \sim - \exp(-1/(c\eps)^2)$, with $\eps, c>0$. Preliminary analysis indicates that an analogous result is true for our nonlocal eikonal equation~\eqref{e:nonloc1}.

\end{enumerate}

One of our main findings is that nonradially symmetric inhomogeneities generate target patterns which are radially symmetric  in the far field. This behavior can be understood heuristically by noticing that when $\eps =0$ the proposed phase equation is invariant under rotations. On the other hand, when $\eps\neq 0$ the inhomogeneity breaks the radial symmetry but only in a region  near its core where its presence is not negligible. Our simulations confirm this behavior showing target patterns that are nonradially symmetric near the origin, but that recover this symmetry in the far field. This behavior can also be explained by the analysis. When we decompose our full equation into polar modes (see Section \ref{s:nonlocalfull}), the only solution that does not decay algebraically is the one that corresponds to the zero-th mode equation. This suggest that what is relevant for the system is the mass of the inhomogeneity and not its shape. This stands in contrasts to the behavior of radially symmetric propagating fronts in bistable reaction-diffusion system for dimensions two or higher \cite{roussier2004stability}. In this case, nonradially symmetric perturbations to the original front persist and solutions lose their original radial symmetry. 

Also, as mentioned in the introduction, the approximations that we have obtained for the frequency of the target patterns are in good agreement with, and in fact refine previous results. However, for values of $\eps$ that we can numerically compute with, they are not accurate enough to enable us to compare the expressions for the frequency in~\eqref{e:omega} with simulations. In ongoing work, we are developing higher order asymptotic solutions to correctly predict and compare the frequency with numerical experiments. 

Finally, the applications that motivated our model equation were large collections of discrete objects, e.g. neurons or particles in a granular medium. Because they involve a large number of interacting units, it seems reasonable to use a continuum model to describe these systems. Our results show that, in general, target wave patterns are found whenever an oscillatory extended system with (possibly nonlocal) diffusive coupling is perturbed by an inhomogeneity (of the appropriate sign). A natural question is to characterize patterns  in   smaller networks, i.e. systems that cannot be approximated as a continuum. This question involves an extra complication, since the topology of the network now becomes very relevant. 

For regular networks, like chains or rings of oscillators with nearest neighbor coupling, it has been shown that a single pacemaker can entrain the whole system \cite{radicchi2006entrainment}. A similar result was found for networks where the oscillators are connected at random \cite{kori2004entrainment}. However, in general, the question of understanding the transition to synchrony and the stability of this solution in networks with different topologies is challenging (see \cite{ arenas2008synchronization} and references therein). For example, in the case of all to all coupling this transition is a function of the coupling strength \cite{kuramoto1975self}, but as connections are removed at random the transition depends now on how many oscillators  are still connected \cite{restrepo2005onset}.  At the same time, new types of synchronous states called chimeras have been discovered in simulations of oscillators with nonlocal coupling, \cite{kuramoto2002coexistence}. These solutions are characterized by having a fraction of the oscillators in phase while the rest are in complete asynchrony (see \cite{yao2013robustness} and references). A similar behavior was also found in a reaction-diffusion system that can be approximated by an integro-differential equation. In this case, chimera states appear as spiral wave solutions with a core that is not in synchrony \cite{shima2004rotating}. 
There is a wealth of such physical phenomena that are yet to be analyzed rigorously, and in the future we hope to further investigate aspects of the central question, namely the interplay between discreteness, nonlocal spatial coupling and ``local" oscillatory behavior.

\section*{Acknowledgements}
Both authors would like to thank Arnd Scheel for many useful discussions and suggestions. G.J. acknowledges the support from the National Science Foundation through the grant DMS-1503115. S.V. acknowledges the support from the Simons Foundation through award 524875.

\appendix

\section{A hierarchy of evolution equations for the phases of coupled oscillators} \label{sec:hierarchy}

A continuum of oscillators with stable limit cycles and small variations in their natural frequencies can be described by the Stuart-Landau equations \cite{kuramoto1984book}
$$
\partial_t z(x,y,t) = i(1+\epsilon^2 g(x,y)) z + (1-|z|^2) z
$$
where $\epsilon \ll 1$, and time has been scaled so that the natural frequencies are all close to 1. We assume the coefficient of the nonlinear term $|z|^2 z$ is real, while full generality allows for a non-zero imaginary part resulting in a dependence of the oscillation frequency on the local amplitude $|z|$ \cite{kuramoto1984book}. Introducing a non-zero imaginary part for this coefficient will not change any of our conclusions, since $|z| \approx 1$ as we argue below, so the frequency is still nearly constant spatially, just ``renormalized" to a different value.

 There are indeed many different ways to spatially couple these oscillators. In order to constrain the possible models, we will impose the following requirements:
\begin{enumerate}
\item If $\epsilon = 0$, the uniform states $z = e^{i t+\phi_0}$ are (neutrally) stable. This is clearly necessary if we hope to describe the $\epsilon \neq 0$ behavior through a slow modulation of the phase, i.e. by replacing the constant $\phi_0$ with a slowly varying (in time)  function $\phi(x,y,t)$.
\item The spatial coupling between the oscillators is weak. Quantitatively, we will assume that the coupling is $O(\epsilon)$.
\item The spatial coupling is isotropic and translation invariant. In particular, this implies that the entire systems is nearly isotropic and translation invariant, and the only place where this symmetry is ``weakly" broken is through explicit spatial dependence of the natural frequencies of the oscillators encoded in the term $\epsilon^2 g(x,y)$.
\item The overall dynamics has a global {\em gauge symmetry} $z(x,y,t) \mapsto e^{i \phi_0} z(x,y,t)$ for any constant $\phi_0$. This symmetry is appropriate for studying slow behavior in systems with fast oscillations (here the natural frequencies $\approx 1$ for the oscillators correspond to the ``fast" time scale) and reflects the fact that the slow behavior is independent of the exact choice of the origin of time (the phase shift $\phi_0$) if it is governed by averaging over many periods of the fast oscillations \cite{aranson_cgle_2002}.
\end{enumerate}
Spatial couplings satisfying these requirements are of the form $ \epsilon \mathcal{L} \star F(|z|^2) z$, where $\mathcal{L} \star$ denotes convolution with a kernel that is isotropic and $O(1)$ as $\epsilon \to 0$. Convolutions are manifestly translation invariant and $F(|z|^2) z$ is the general form of a gauge invariant term. Finally, in order to maintain the neutral stability of the uniform states $z = e^{i t+\phi_0}$ for $\epsilon = 0$, we will need that $\mathcal{L} \star 1 = 0$, i.e. constants are in the kernel of $\mathcal{L}$. Taken together, this suggests that we should consider models of the form
$$
\partial_t z = i(1+\epsilon g) z + (1-|z|^2) z + \epsilon(\mathcal{L}_1 \star F_1(|z|^2) z + \mathcal{L}_2 \star F_2(|z|^2) z + \ldots + \mathcal{L}_n \star F_n(|z|^2) z)
$$
This is the appropriate generalization of the complex Ginzburg-Landau equation \cite{kuramoto1984book,aranson_cgle_2002} to our context of nonlocally coupled oscillators whose frequencies are amplitude independent. Following the approach in \cite{kuramoto1984book}, we write $z(x,y,t) = r(x,y,t)e^{i t + i \phi(x,y,t)}$ and expect that $r(x,y,t) \approx 1$ (corresponding to the natural dynamics of the oscillators) and $|\partial_t \phi| \ll 1$ (since we have already accounted the dominant part of the oscillatory behavior $z \sim e^{it}$). We will call $\phi$ the relative phase of the oscillators. While the relative phase varies slowly in time, {\em we make no a priori assumptions about its spatial variation}. The above arguments suggest that we can replace $|z|^2$ by 1 in the above expressions to get
\begin{align*}
\mathcal{L} & = F_1(1) \mathcal{L}_1 +  F_2(1) \mathcal{L}_2+ \cdots +  F_n(1) \mathcal{L}_n \\
\partial_t r + i \partial_t \phi & = i \epsilon^2 g + (1-|r|^2)r + \epsilon e^{-i \phi} \mathcal{L} \star r e^{i \phi}.
\end{align*}
 Separating the real and imaginary parts gives the evolution equations for $r$ and $\phi$. Our assumptions imply that the dominant part of the $r$ dynamics is $\partial_t r \approx r(1-r^2)$ so that the system relaxes to its {\em slow manifold} $r \approx 1$ {\em quickly}, i.e on the time scale of the natural oscillations. Indeed, $r$ is slaved to $\phi$ in this regime and we can solve by linearizing $r(1-r^2) \approx - 3(r-1)$ near $r = 1$ to obtain
 $$
 r \approx 1 + \frac{\epsilon}{3} \mathrm{Re}\left[e^{-i \phi} \mathcal{L} \star e^{i \phi}\right] \approx 1.
 $$
 To this order of approximation, the evolution equation for the phase is therefore
 \begin{equation}
  \partial_{\tau} \phi = \epsilon g +  \mathrm{Im}\left[e^{-i \phi} \mathcal{L} \star e^{i \phi}\right],
  \label{eq:nonlocal_schrodinger}
 \end{equation}
where $\tau = \epsilon t$ is the natural slow time that arises from balancing the time derivative of $\phi$ with the nonlocal coupling. This equation can also be recast as $\partial_{\tau} \phi(\mathbf{x}) = \epsilon g(\mathbf{x}) + \int K_R(|\mathbf{x}-\mathbf{x'}|) \sin(\phi(\mathbf{x}') - \phi(\mathbf{x})) d \mathbf{x}' +  \int K_I(|\mathbf{x}-\mathbf{x'}|) \cos(\phi(\mathbf{x}') - \phi(\mathbf{x})) d \mathbf{x}'$ by splitting into real and imaginary parts.  This is a variant of a well known equation in the context of neural field models (see~\cite{schwemmer2012theory} and references therein).
 
 We now specialize~\eqref{eq:nonlocal_schrodinger} to particular types of spatial coupling $\mathcal{L}$. The class of translation invariant, isotropic, differential operators of finite order that annihilate the constants are given by 
$$
\mathcal{L} \star z = (a_1+ i b_1) \Delta z + (a_2+ ib_2) \Delta^2 z + \cdots + (a_n+ ib_n) \Delta^n z
$$
where $\Delta = \partial_x^2 + \partial_y^2$ is the Laplacian. This motivates the definitions
$$
\Gamma_k[\phi] = \mathrm{Im}\left[e^{-i \phi} \Delta^k e^{i \phi}\right], \quad \Sigma_k[\phi] = \mathrm{Re}\left[e^{-i \phi} \Delta^k e^{i \phi}\right], \quad k = 1,2,3,\ldots
$$
In terms of these operators, the slow evolution of the relative phase $\phi$ is given by
\begin{equation}
\partial_{\tau} \phi = \epsilon g + a_1 \Gamma_1[\phi] + b_1 \Sigma_1[\phi] + \cdots + a_n \Gamma_n[\phi] + b_n \Sigma_n[\phi] 
\label{eq:master}
\end{equation}
These evolution equations, for different choices of $a_i,b_i$, are {\em universal} models for the slow time evolution of the relative phases of coupled oscillators. We can compute the first few $\Gamma_k$ and $\Sigma_k$ explicitly to obtain 
\begin{align*}
\Gamma_1[\phi] & = \Delta \phi \\
\Sigma_1[\phi] & = -|\nabla \phi|^2 \\
\Gamma_2[\phi] & = \Delta^2 \phi - 2 |\nabla \phi|^2 \Delta \phi - 4 \nabla \phi \cdot \nabla \nabla \phi \nabla \phi \\
\Sigma_2[\phi] & = |\nabla \phi|^4 - 4 \nabla \phi \cdot \nabla(\Delta \phi) -2 |\nabla \nabla \phi|^2 - (\Delta \phi)^2
\end{align*}
where $\nabla \nabla \phi$ is the $2 \times 2$ Hessian matrix. The expressions for $\Gamma_k$ and $\Sigma_k$ quickly get very complicated as $k$ increases,. We therefore record a few observations on the structure of $\Gamma_k$ and $\Sigma_k$ that are easy to show by an inductive argument using the recursion $\displaystyle{\Sigma_k + i \Gamma_k = e^{-i \phi} \Delta \left[ (\Sigma_{k-1}[\phi] + i \Gamma_{k-1}[\phi])e^{i \phi}\right]}$
\begin{enumerate}
\item Every term in $\Gamma_k$ and $\Sigma_k$ has a total of $2k$ spatial derivatives of $\phi$.
\item $\Gamma_k$ has terms that are homogeneous with degrees $1,3,\ldots,2k-1$ in $\phi, \nabla \phi, \nabla \nabla \phi, \ldots$ and $\Sigma_k$ has terms that are homogeneous with degrees $2,4,\ldots,2k$.
\item $\Gamma_k[\phi] = \Delta^k \phi +$ h.o.t, $\Sigma_k[\phi]$ is an isotropic sum/difference of squares of $k$th order spatial derivatives of $\phi$ + h.o.t, where h.o.t. refers to terms that are higher order in the homogeneity in $\phi$. 
\end{enumerate}
We can obtain the {\em dispersion relation} by substituting a plane wave ansatz $\phi = k_x x + k_y y - \omega(k_x,k_y) \tau$ into~\eqref{eq:master} with $\epsilon = 0$ to obtain
$$
 \omega = b_1 k^2 - b_2 k^4 + \cdots
$$
where $k^2 = k_x^2+k_y^2$. 
We will henceforth restrict ourselves to  the case where the long-wave dispersion coefficient $\displaystyle{\left.\frac{d^2 \omega}{d k^2}\right|_{k = 0} = 2b_1 > 0}$. We define a hierarchy of evolution equations by retaining the linear and quadratic terms in $\Gamma_k$ and $\Sigma_k$. The first two equations in this hierarchy are
\begin{align}
\partial_{\tau} \phi & = \epsilon g + a_1 \Delta \phi - b_1 |\nabla \phi|^2 \nonumber \\
\partial_{\tau} \phi & = \epsilon g + a_1 \Delta \phi + a_2 \Delta^2 \phi - b_1 |\nabla \phi|^2 - b_2 (2 |\nabla \nabla \phi|^2 + (\Delta \phi)^2) 
\label{eq:hierarchy}
\end{align}

A basic consistency check for our procedure is that the first equation in the hierarchy is indeed the usual viscous eikonal equation.

A useful approach in modeling physical phenomena is to retain just as many terms in the model as necessary to obtain the desired behavior in the solutions. For example, this principle leads to the Kuramoto-Sivashinsky equation, where the term $a_2 \Delta^2 \phi$ is retained if $a_1 < 0$, but the quadratic terms involving four derivatives $-b_2 (2 |\nabla \nabla \phi|^2 + (\Delta \phi)^2)$ are nonetheless dropped although they are formally of the same order in derivatives.  While very successful, the logic of this approach is somewhat circular. 

We take a slightly different modeling approach. Our goal is to get the right small $k$ (long-wavelength) behavior. To this end, we will approximatet he quadratic terms involving the sums/differences of squares by a single term of the form $- |\mathcal{J} \ast \nabla \phi|^2$ that captures the long-wave  behavior of these nonlinearities. That is, we will require that, for an exponent $\alpha$ that is as large as possible, and for all compactly supported smooth function $\psi \geq 0$, we have
$$
\int \left\{ b_1 |\nabla \phi|^2 + b_2 (2 |\nabla \nabla \phi|^2 + (\Delta \phi)^2) -  |\mathcal{J} \ast \nabla \phi|^2 \right\} \psi\left(\frac{x}{L},\frac{y}{L}\right) \, dx dy \sim O(L^{-\alpha}) \mbox{ as } L \to \infty.
$$
For any function $g \in L^1(\mathbb{R}^2)$, the dominated convergence theorem implies that, as $L \to \infty$,  $\displaystyle{\int g(x,y)  \psi\left(\frac{x}{L},\frac{y}{L}\right) \, dx dy \to \psi(0) \hat{g}(0})$ where $\hat{g} = \mathcal{F}(g)$ denotes the Fourier transform of $g$. Consequently, for $\phi \in C_0(\R^2)$, our requirement that the approximation should capture the right long-wave behavior of the nonlinearity can be written as
\[ \mathcal{F}\left(b_1 |\nabla \phi|^2 + b_2 (2 |\nabla \nabla \phi|^2 + (\Delta \phi)^2) \right)(0) = \mathcal{F}\left(|\mathcal{J} \ast \nabla \phi|^2\right)(0)
\]
We now show that we can indeed find an expression for the Fourier symbol corresponding to the convolution operator $\mathcal{J}$ satisfying this condition. 

With ${\bf k}= (k_1,k_2)$ representing our variable in Fourier space and defining the nonlinear operator $N$ by 
 \begin{align*}
  N( \hat{\phi}) = &\mathcal{F}\Big( b_1(\phi_x^2 + \phi_y^2)  + 2 b_2(  \phi_{xx}^2 + 2\phi_{xy}^2 + \phi_{yy}^2)+ b_2( \phi_{xx}^2+ 2\phi_{xx}\phi_{yy} + \phi_{yy}^2)  \Big) \\
    = & -b_1 ( k_1 \hat{\phi} \ast  k_1 \hat{\phi} +  k_2 \hat{\phi} \ast k_2 \hat{\phi}) + 3b_2 (k^2_1 \hat{\phi} \ast  k^2_1 \hat{\phi} + k_2^2 \hat{\phi} \ast  k_2^2 \hat{\phi})\\
  & + 4b_2( k_1k_2 \hat{\phi} \ast  k_1k_2 \hat{\phi}) + 2b_2(k_1^2 \hat{\phi} \ast k_2^2 \hat{\phi}) ),
  \end{align*}
a straight forward computation then shows that
\begin{align*}
 N( \hat{\phi})(0) = &\int_{\R^2} (b_1 + 3b_2 | {\bf q}|^2)|{\bf q} |^2 \hat{\phi}(- {\bf q})\hat{\phi}( {\bf q})\;d {\bf q}\\
    N( \hat{\phi})(0)=& \int_{\R^2} J(-{\bf q})  \hat{\phi}(- {\bf q}) \langle \rmi q_1 , \rmi q_2  \rangle \cdot \langle -\rmi q_1, -\rmi q_2 \rangle J( {\bf q})\hat{\phi}( {\bf q}) \; d {\bf q},
   \end{align*}
 where $J( {\bf k}) = \sqrt{b_1 + 3b_2 | {\bf k}|^2}$. Since this holds for all $\phi \in C_0(\R^2)$, it is clear then that the Fourier symbol for $\mathcal{J}$ can be picked to be precisely the multiplication operator $J( {\bf k})$.

An inductive argument allows us to extend these ideas to higher order equations in the hierarchy (cf.~\eqref{eq:hierarchy}), {\em viz.} truncating~\eqref{eq:master} at quadratic order, followed by a reduction of the quadratic terms as above will give a reduced model
$$
\partial_{\tau} \phi = \sum a_j \Delta^j \phi - |\mathcal{J} \ast \nabla \phi|^2,
$$
where the kernel $\mathcal{J}$ has a Fourier symbol of the form 
$$
J({\bf k}) = \left[\sum_j c_j b_j |{\bf k}|^{2(j-1)}\right]^{1/2},
$$
and the constants $c_j$ are universal (i.e. independent of the differential operator $\mathcal{L}$) and can be computed explicitly. In particular, $c_1 = 1, c_2 = 3$.

If we assume that the function $e^{i \phi}$ is band-limited, a physically reasonable assumption precluding the presence of arbitrarily small-scale structures in the pattern of the relative phase, we can extend the above argument to reduce~\eqref{eq:master} for general (i.e. not necessarily a finite order differential operators) $\mathcal{L}$ and $\mathcal{J}$ with continuous Fourier symbols $L = L(|{\bf k}|)$, $J = J(|{\bf k}|)$ that are also exponentially localized. Indeed observing that the space $C_0(\R^2)$ is dense in weighted Sobolev spaces,  an argument by approximation gives a reduction of~\eqref{eq:master} to 
$$
\partial_{\tau} \phi = \epsilon g + \mathcal{L \ast \phi} - |\mathcal{J} \ast \nabla \phi|^2.
$$
By rescaling $\tau$ and $\phi$ we can, W.L.O.G, set $a_1=1, b_1 = 1$, so that $L(0) = 1, J(0) = 1$. This is the basic model (Eq.~\eqref{e:nonloc1}) that we consider in the body of the paper.

\section{Proofs of subsidiary results} \label{sec:proofs}

\subsection{Algebraic decay in Kondratiev spaces}\label{s:appendixdecay}

\begin{Lemma*}{ \bf \ref{l:decay}.}
Let $f \in M^{1,2}_{\gamma}(\R^d)$ then $|f({\bf x})| \leq C \|f\|_{M^{1,2}_{\gamma}(\R^d)} (1+ |{\bf x}|^2)^{-(\gamma+d/2)}$ as $|{\bf x}| \rightarrow \infty$.
\end{Lemma*}
\begin{proof}
Let $(\theta, r)$ represent spherical coordinates in $d$ dimensions, with $r$ being the radial direction and $\theta$ representing the coordinates in the unit sphere, $\Sigma$. Then 
\begin{align*}
 \int_\Sigma |f(\theta,r) |^2\; d\theta & \leq \int_\Sigma \left( \int_{\infty}^R |\partial_r f (\theta,s) | \;ds \right)^2 \;d\theta\\[2ex]
 & = \int_\Sigma \left (  \int_{\infty}^R s^{\alpha} s^{\gamma+1} |\partial_r f (\theta,s) | s^{(d-1)/2}\;ds  \right)^2\;d\theta\\[2ex]
 & \leq \int_\Sigma \left (\int_{\infty}^R s^{2 \alpha} \right)^2 \left(\int_{\infty}^R  s^{2(\gamma+1)} |\partial_r f (\theta,s) |^2 s^{(d-1)}\;ds \right)^2\;d\theta \\[2ex]
 & \leq R^{2\alpha +1} \|\nabla f\|^2_{L^2_{\gamma+1}(\R^d)}
 \end{align*}
where $\alpha =  -(\gamma+1)+ (1-d)/2$ and $R$ is fixed. We therefore have the following inequality
\[ \|f(\cdot, R)\|_{L^2} \leq R^{-(\gamma+d/2)}  \|\nabla f\|_{L^2_{\gamma+1}(\R^d)} \]

From Adams and Fournier \cite[Theorem, 5.9]{adams2003sobolev} we know that there exist a constant $C$ such that
\[ \| f(\cdot, R) \|_{\infty} \leq C \| f(\cdot,R)\|_{W^{m,p}(\R^d)}^\delta \|f(\cdot,R) \|_{L^q(\R^d)}^{(1-\delta)}\]
where $\delta = dp/ ( dp + (mp -d)q)$ and $1/q + 1/p = 1$ . Choosing $p=q=2$  and $m=0$, we obtain

\[ \| f(\cdot, R) \|_{\infty} \leq C \|f(\cdot,R) \|_{L^2} \leq C R^{-(\gamma+d/2)}  \|\nabla f\|_{M^{1,2}_{\gamma}(\R^d)} \]

\end{proof}

\subsection{Fredholm properties of $\Delta^m(\Id-\Delta)^{-l}$}\label{s:appendixb}
In this section we look at the Fredholm properties of the operator \[ \begin{array}{c c c}
\Delta^m(\Id-\Delta)^{-l}:  \mathcal{D} \subset L^p_{\gamma- 2m} (\R^d) & \longrightarrow & L^p_\gamma(\R^d)\\
  \end{array}\]
 with domain $\mathcal{D} = \{ u({\bf x}) \in L^p_{\gamma-2m}(\R^d) \mid (\Id-\Delta)^{-l} u({\bf x}) \in M^{2m,p}_{\gamma-2m}(\R^d)\}$. 
 More precisely, we prove the following proposition.
\begin{Proposition*}{\bf \ref{p:fredholmNormalF}.}
Let $\gamma \in \R$, $\alpha, \beta \in \Z^+ \bigcup \{ 0\}$, $m,d \in \Z^+$, and $l \in \Z$. Then, the operator, 
\[\Delta^m(\Id-\Delta)^{-l}:  \mathcal{D} \subset L^p_{\gamma- 2m} (\R^d)  \longrightarrow  L^p_\gamma(\R^d),\]
with domain $\mathcal{D}=\{u \in L^p_{\gamma-2m}(\R^d) \mid (\Id-\Delta)^{-l}u \in M^{2m,p}_{\gamma-2m}(\R^d) \}$  
\begin{itemize}
\item is a Fredholm operator for $\alpha+d/p<\gamma< -\beta -d/p +2m$  with kernel and cokernel given by
\[ \ker = \bigcup_{j=0}^\beta \mathcal{H}_{j,k}, \quad  \coker = \bigcup_{j=0}^\alpha \mathcal{H}_{j,k};\]
 \item and not Fredholm for values of  $\gamma \in \{ j + d/p : j \in \Z\}$. 
\end{itemize}
\end{Proposition*}

The proof consists in showing that the operator $\Delta^m(\Id-\Delta)^l$ is the composition of 
\begin{enumerate}
\item \label{item:invertible} an invertible operator $ (\Id-\Delta)^{-l} : \mathcal{D}\subset L^p_{\gamma-2}(\R^d) \longrightarrow M^{2m,p}_{\gamma-2m}(\R^d) \subset W^{2m,p}_{\gamma-2m}(\R^d)$
\item \label{item:fred} and a Fredholm operator $\Delta^m: M^{2m,p}_{\gamma-2}(\R^d) \longrightarrow L^p_{\gamma}(\R^d).$
\end{enumerate}

Item (\ref{item:defined}.) follows from Lemma \ref{l:isomorphism}, and the result of item (\ref{item:fred}.) from Lemma \ref{l:Laplacefredholm}. Then, the span of the kernel and cokernel of $\Delta^m(\Id-\Delta)^l$ can be determined using Lemma \ref{l:kernel} and duality.

We start by showing the operator $(\Id-\Delta)$ is an isomorphism in weighted spaces. As part of the proof we need the following Theorem by Kato, see \cite{kato2013perturbation}.

\begin{Theorem*}[Kato, p.370]
Let $T(\gamma)$ be a family of compact operators in a Banach space $X$ which are holomorphic for all $ \gamma \in \C$. Call $\gamma$ a singular point if 1 is an eigenvalue of $T(\gamma)$. Then either all $\gamma \in D$ are singular points or there are only finitely many singular points in each compact subset of $D$.
\end{Theorem*}

With the notation $\langle {\bf x} \rangle = ( 1+ | {\bf x}|^2)^{1/2}$, we are now ready to show that:
\begin{Lemma}\label{l:isomorphism}
Given $s \in \Z, p \in (1,\infty), \gamma \in \R$, the operator $(\Id-\Delta): W^{s,p}_{\gamma}(\R^d) \rightarrow W^{s-2,p}_{\gamma}(\R^d)$ is an isomorphism.
\end{Lemma}
\begin{proof}
We show the result only for $s\geq 2$, since the results can be extended to all $s \in \Z$ by duality. We have the following commutative diagram
\begin{center}
\begin{tikzcd}[column sep = large, row sep= huge]
W^{s,p}_\gamma(\R^d) \arrow[r, "(\Id-\Delta)"] \arrow[d, "\langle {\bf x} \rangle^\gamma"] 
& W^{s-2,p}_\gamma(\R^d) \arrow[d, "\langle {\bf x} \rangle^\gamma"] \\
W^{s,p}(\R^d) \arrow[r, "A(\gamma)"] 
& W^{s-2,p}(\R^d)
\end{tikzcd}
\end{center}
where $A(\gamma)u = (\Id- \Delta) u + T(\gamma) u$, and $T(\gamma) u= -\gamma(\gamma+2)|{\bf x}|^2 \langle {\bf x} \rangle^{-4} u + \gamma \langle {\bf x} \rangle^{-2} u  + 2\gamma \langle {\bf x} \rangle^{-2} {\bf x} \cdot \nabla u$ . The operator $T(\gamma)$ may be approximated using compactly supported functions. It then follows by the Rellich Kondrachov embedding theorem that $T(\gamma)$ is a compact perturbation of $(\Id-\Delta): W^{s,p}(\R^d) \rightarrow W^{s-2,p}(\R^d)$. 

The results of the lemma follow if we can show that $\ker A(\gamma) =\{0\}$ for all $\gamma \in \R$. Suppose for the moment that this is not true and that we can find a number $\gamma^*\in \R$ and a function $u \in W^{s,p}(\R^d)$ such that $A(\gamma^*)u =0$. Then, from the commutative diagram and the embedding $W^{s,p}_{\gamma^*}(\R^d) \subset W^{s,p}_{\gamma}(\R^d)$, with $\gamma< \gamma^*$, we obtain that $A(\gamma)u = 0$ for all $\gamma< \gamma^*$.

On the other hand, the operator 
\[
\begin{array}{c c c}
W^{s,p}(\R^d) & \longrightarrow &W^{s-2,p}(\R^d)\\
u &\longmapsto  & (\Id-\Delta)^{-1} T(\gamma) u,
\end{array}\]
which is  compact and analytic for all $\gamma \in \C$, has $\lambda =1$ as an eigenvalue if and only if  $A(\gamma)$ has a non trivial kernel. We have therefore shown that $\lambda=1$ is an eigenvalue of $(\Id-\Delta)^{-1} T(\gamma)$ for all $\gamma<\gamma^*$. Now, Kato's Theorem implies that $\lambda =1$ is an eigenvalue of $(\Id-\Delta)^{-1} T(\gamma)$ for all $\gamma \in \C$. In particular, this result holds form $\gamma=0$, which is a contradiction as $(\Id-\Delta)$ is invertible. It therefore follows that $\ker A(\gamma) =\{0\}$ for all $\gamma \in \R$ and that $(\Id-\Delta): W^{s,p}_{\gamma}(\R^d) \rightarrow W^{s-2,p}_{\gamma}(\R^d)$ is an isomorphism.

\end{proof}

Before showing the Fredholm properties of $\Delta^m$ in Kondratiev spaces we need some results concerning homogenous polynomials. For more detailed proofs regarding the following results see \cite{stein2016introduction}.

{\bf Homogenous Polynomials:} Let $\bbP_m$ denote the space of complex valued homogenous polynomials in $\R^d$ of degree $m$, and let $\bbH_m \subset \bbP_m$ denote the subspace of harmonic polynomials. One can define a inner product in $\bbP_m$, via $\langle P, Q\rangle = P(D)\bar{Q}$, where the bar  denotes complex conjugation and $P(D)$ is the differential operator by which $x_i$ is replaced with $\partial/\partial_j$. With this notion of an inner product it is possible to show that:

\begin{Lemma}\label{l:factorization}
The space  $\bbP_m$ is a direct sum of the form
\[\bbP = \bbH_m \oplus |{\bf x} |^2 \bbH_{m-2} \oplus |{\bf x} |^4 \bbH_{m-4} \oplus \cdots \oplus |{\bf x} |^{2 \alpha} \bbH_{m -2\alpha},\]
where $\alpha$ is a positive integer such that $m =2 \alpha$ if $m$ is even, or $ m+1 = 2 \alpha$ if $m$ is odd.
\end{Lemma}

Since the dimension of $\bbP_m$ is given by $d_m = { n +m -1 \choose m}$, it follows from Lemma \ref{l:factorization} that the $\dim (\bbH_m ) ={ n +m -1 \choose m} - { n +m -3 \choose m-2}.$

The following notation,
\[
\mathcal{H}_{m,k} = \left\{
\begin{array}{l c c}
  \bbP_m &\mbox{if} & m< 2k\\
 \bbH_m \oplus |{\bf x} |^2\bbH_{m-2} \oplus \cdots \oplus |{\bf x} |^{2(k-1)}\bbH_{m -2(k-1)} & \mbox{if} & m\geq 2k
 \end{array}
 \right. \]
is used in the next lemma.
\begin{Lemma}\label{l:kernel}
The operator $\Delta^k:\bbP_m \rightarrow \bbP_{m-2}$ is onto with kernel $\mathcal{H}_{m,k}$
\end{Lemma}

\begin{proof}
The result is trivially true if $m \leq 2k$. For the case $m>2k$ suppose the result does not hold, so that there exist a $q({\bf x} ) \in \bbP_{m-2k}$ such that $\langle \Delta^k p, q \rangle =0$ for all $p({\bf x} ) \in \bbP_m$. Take $p({\bf x} ) = |{\bf x} |^{2k} q({\bf x} )$ and notice that 
\[ 0 = \overline{\langle \Delta^k p,q \rangle }= \langle q, \Delta^k p \rangle = q(D) \Delta^k \overline{p} = \Delta^k q(D)\bar{p} = \langle p, p \rangle.\]
Since $\langle \cdot, \cdot \rangle$ is an inner product, it follows that $p({\bf x} ) \equiv 0$ and as a consequence the map $\Delta^k:\bbP_m \rightarrow \bbP_{m-2k}$ is onto.

Now, let $r({\bf x} ) = |{\bf x} |^{2k} q({\bf x} )$  for some $q({\bf x} ) \in \bbP_{m-2k}$, and take $p({\bf x} ) \in \bbP_m$ such that $\langle r,p \rangle =0$. This last equality holds if and only if
\[ \langle r,p \rangle = r(D) \bar{p} = q(D) \Delta^k \bar{p} =\langle q, \Delta^k p \rangle =0. \]
That is, if and only if $p({\bf x} ) \in \ker \Delta^k$. Therefore, $\bbP_m = \ker(\Delta^k) \oplus |{\bf x} |^{2k} \bbP_{m-2k}$, and by Lemma \ref{l:factorization} we must have
\begin{align*}
 \ker(\Delta^k) &= \bbH_m \oplus |{\bf x} |^2\bbH_{m-2} \oplus \cdots \oplus |{\bf x} |^{2(k-1)}\bbH_{m -2(k-1)} \\
 &=\mathcal{H}_{m,k}.
 \end{align*}

\end{proof}

We are now ready to state the Fredholm properties of $\Delta^m:M^{2m,p}_{\gamma-2m}(\R^d) \rightarrow L^p_{\gamma}(\R^d)$.

\begin{Lemma}\label{l:Laplacefredholm}
Given $\gamma \in \R$, $\alpha, \beta \in \N \bigcup \{ 0\}$, $m, n \in \Z^+$, and $p \in (1,\infty)$ the operator
\[\begin{array}{c c c}
 M^{2m,p}_{\gamma-2m} (\R^d)& \longrightarrow & L^p_{\gamma}(\R^d)\\
u & \longmapsto &\Delta^m u
\end{array}
\]
is  
\begin{itemize}
\item Fredholm for $\alpha+n/p<\gamma< -\beta -n/p +2m$ with kernel and cokernel given by
\[ \ker = \bigcup_{j=0}^\beta \mathcal{H}_{j,k}\quad \coker = \bigcup_{j=0}^\alpha \mathcal{H}_{j,k};\]
\item and it is not Fredholm for  values of $\gamma \in \{ j + n/p : j \in \Z\}$.
\end{itemize}
\end{Lemma}

\begin{proof}
From the results in \cite{mcowen1979behavior}, it is straight forward to see that the operator is Fredholm. The span of the kernel and cokernel follows from Lemma \ref{l:kernel}, while the range of values for $\gamma$ can be found by determining when the subspace $\bbP_j $ is contained in $ \in L^p_{\gamma-2m}(\R^d)$, or in  $L^p_{-\gamma}(\R^d)$.
\end{proof}

\subsection{Fredholm properties of radial derivatives} \label{s:appendixRadialderivatives}
In what follows, $\gamma \in \R$ and $p,q \in (1,\infty)$ are conjugate exponents. We also use the notation $\langle r \rangle = (1+r^2)^{1/2}$.
\begin{Lemma}\label{l:invertible}
Let $\gamma \in \R$, and $p \in (1, \infty)$. Then  the operator 
\[
\begin{array}{c c c}
W^{1,p}_{r,\gamma}(\R^2)& \longrightarrow &L^p_{r,\gamma}(\R^2)\\
u & \longmapsto&  \partial_ru + \frac{1}{r}u -u 
\end{array}
\]
 is an invertible operator.
\end{Lemma}
\begin{proof}
We show that the inverse operator
\[
\begin{array}{c c c}
L^p_{r,\gamma}(\R^2) & \longrightarrow & W^{1,p}_{r,\gamma}(\R^2)\\
f(r) & \longmapsto&   u(r) = \frac{1}{r} \int_{\infty}^r \rme^{-(r-s)}f(s) s \; ds
\end{array}
\]
is a bounded operator using the following inequality
\[ \| u\|_{L^p_\gamma (\R^2)} \leq \| u\|_{L^p_\gamma (B_1)} + \| u\|_{L^p_\gamma (\R^2 \backslash B_1)},\]
where $B_1$ is the unit ball in $\R^2$. 
First, given $f \in C^{\infty}_0$ a simple calculation shows that  $u(r)$  has the same type of singularity as $r f(r)$ near the origin. Since $C^{\infty}_0(\R)$ is dense in $L^p_\gamma(\R^2)$, the result also holds for any $f \in L^p_\gamma(\R^2)$ and consequently,
\[ \|u\|_{L^p_\gamma(B_1)} \leq C \|f\|_{L^p_\gamma(\R^2)},\]
as well as
\[ \|u(r)/r\|_{L^p_\gamma(B_1)} \leq C \|f\|_{L^p_\gamma(\R^2)}.\]

On the other hand, because $\langle r \rangle ^\eta \langle s \rangle^{-\eta} \leq \langle r-s \rangle^{|\eta|}$ holds for any $\eta\in \R$ we obtain
\begin{align*}
 \|u\|^p_{L^p_\gamma(\R^2 \backslash B_1)}  & = \int_1^{\infty} \left |\frac{1}{r} \int_\infty^r \rme^{(r-s) }f(s) s \;ds  \right|^p \langle r \rangle^{\gamma p} r \;dr \\
 & \leq  \int_1^{\infty} \left | \int_\infty^r \rme^{(r-s) } \langle r-s \rangle^{|\gamma - 1 + 1/p|} f(s) \langle s \rangle^{\gamma +1/p} \;ds  \right|^p  \;dr,
\end{align*}
from which we deduce using Young's inequality that $ \|u\|_{L^p_\gamma (\R^2 \backslash B_1)} \leq C(\gamma) \|f \|_{L^p_\gamma (\R^2)}$.  It is then straightforward to see that $\| \partial_r u \|_{L^p_\gamma (\R^2)} \leq \|f \|_{L^p_\gamma (\R^2)}$, since $\partial_r u = u - \frac{1}{r} u - f$.

\end{proof}

The next propositions states the Fredholm properties for $\partial_r$ and its adjoint $\partial_r + 1/r$.
\begin{Proposition}\label{p:dr}
Given $p \in (1,\infty)$ the operator $\partial_r :M^{k,p}_{r,\gamma-1}(\R^2) \rightarrow L^p_{r,\gamma}(\R^2)$ is a Fredholm operator and
\begin{itemize}
\item for $\gamma>1 -2/p$ it is invertible, whereas 
\item for $\gamma<1 -2/p$ it is surjective with $\ker =\{1\}$.
\end{itemize}
On the other hand, the operator does not have closed range for $\gamma = 1-2/p$.
\end{Proposition}

Similarly, 
\begin{Proposition}\label{p:dr1overr}
Given $p \in (1,\infty)$, the operator $\partial_r +\frac{1}{r}:M^{k,p}_{r,\gamma-1}(\R^2) \rightarrow L^p_{r,\gamma}(\R^2)$ is a Fredholm operator and
\begin{itemize}
\item for $\gamma> 2-2/p$ it is injective with $\coker =\{1\}$, whereas
\item for $\gamma < 2-2/p$ it is invertible.
\end{itemize}
On the other hand, the operator does not have closed range for $\gamma = 2-2/p$.

\end{Proposition}

The proof of the above propositions follows from the next lemmas and duality.

\begin{Lemma}\label{l:dr}
Given $\gamma > 1- 2/p $ and $p \in (1,\infty)$, the operator $\partial_r: M^{1,p}_{r,\gamma-1}(\R^2) \rightarrow L^p_{r,\gamma}(\R^2)$ is invertible.
\end{Lemma}

\begin{proof}
We define the inverse operator
\[ \begin{array}{c c l}
\partial_r^{-1}: L^p_{r,\gamma}(\R^2) & \rightarrow & M^{1,p}_{r,\gamma-1}(\R^2)\\
f(r) &\longmapsto & u(r) = \int_\infty^r f(s) \; ds
\end{array}
\]

and show, using the following inequality 
\[ \| u\|_{L^p_{\gamma-1} (\R^2)} \leq \| u\|_{L^p_{\gamma-1} (B_1)} + \| u\|_{L^p_{\gamma-1} (\R^2 \backslash B_1)},\]
where $B_1$ is the unit ball in $\R^2$,  that $\|u\|_{L^p_{\gamma-1} (\R^2)} < \|f\|_{L^p_\gamma (\R^2)}$. 

First, the result  $\|u\|_{L^p_{\gamma-1} (B_1)} \leq \|f\|_{L^p_\gamma (B_1)}$ is a consequence of the Sobolev embeddings and the fact that $\langle r \rangle$ is a bounded function on $B_1$. 

To show that  $\| u\|_{L^p_{\gamma-1} (\R^2 \backslash B_1)} \leq \| f \|_{L^p_{\gamma-1} (\R^2 \backslash B_1)}$  consider the following scaling variables 
\[ \tau = \ln(r), \quad r \in [1,\infty), \quad w(\tau) = u(\rme^\tau) \rme^{\sigma \tau}, \quad g(\tau) = f(\rme^{\tau}) \rme^{(\sigma+1)\tau}\]
and notice that $w(\tau)$ satisfies $\partial_\tau (w \cdot \rme^{-\sigma \tau}) = \rme^{-\sigma \tau} g(\tau)$, hence $w(\tau) = \int_{\infty}^\tau \rme^{\sigma (\tau -s)} g(s) \;ds$. Then, letting $\sigma = \gamma - 1 + 2/p>0$ and using Young's inequality we arrive at 
\[ \| u\|_{L^p_{\gamma-1}(\R^2\backslash B_1) }= \| w\|_{L^p [1,\infty)} \leq \frac{1}{\sigma} \|g\|_{L^p[1,\infty)} \leq C \|f\|_{L^p_\gamma (\R^2\backslash B_1)}.\]

\end{proof}

\begin{Lemma}\label{l:dr1overr}
Given $\gamma> 2-2/p$ and $p \in (1,\infty)$, the operator $\partial_r + \frac{1}{r}: M^{1,p}_{r,\gamma-1}(\R^2) \rightarrow L^p_{r,\gamma}(\R^2)$ is Fredholm index $i =-1$ with $\coker = \{ 1\}$.
\end{Lemma}
\begin{proof}
Since for $\gamma > 2-2/p$ the function $1$ is in the dual of $L^p_{r,\gamma}(\R^2)$ and it defines a bounded linear functional. Therefore, the space 
\[L^p_{r,\gamma,\perp} = \{ u \in L^p_{r,\gamma}(\R^2) \mid \int_0^\infty u \cdot r \;dr =0\}\]
 is a closed subspace. Moreover, given $u \in C^{\infty}_0 \subset M^{1,p}_{\gamma-1}$ integration by parts shows that  $\partial_r u +\frac{1}{r} u \in L^p_{r,\gamma,\perp}$. Then, because $\gamma>2-2/p$ and $C^\infty_0 $ is dense in $M^{1,p}_{\gamma-1}$ the same is true for any $ u \in M^{1,p}_{r, \gamma-1}$. The results of the lemma then follow if one shows that the inverse operator
 \[ \begin{array}{c c l}
 L^p_{r,\gamma,\perp} & \rightarrow & M^{1,p}_{r,\gamma}(\R^2)\\
 f(r) & \longmapsto & u(r) = \frac{1}{r} \int_{\infty}^r f(s) \cdot s \;ds
 \end{array}
 \]
 is bounded. The proof of this last statement follows a similar arguments as in Lemma \ref{l:invertible} and Lemma \ref{l:dr}, and we therefore omit it.
 
 \end{proof}
 
 To complete the proof of Propositions \ref{p:dr} and \ref{p:dr1overr} we need to define the extended operators $\partial_r: L^p_{r,\gamma-1}(\R^2) \rightarrow M^{-1,p}_{\gamma}(\R^2)$ and $\partial_r+\frac{1}{r}: L^p_{r,\gamma-1}(\R^2) \rightarrow M^{-1,p}_{r,\gamma}(\R^2)$ defined via 
 \[ \partial_r u (v) = \llangle u, (\partial_r)^* v \rrangle = \llangle u, (\partial_r + 1/r) v \rrangle, \qquad \forall u \in L^p_{r,\gamma-1}(\R^2), \forall v \in M^{1,q}_{r,-\gamma}(\R^2)\]
where the double brackets $\llangle u, v \rrangle$ denote the paring between an element $v \in X$ and a linear functional $u \in X^*.$ Notice as well that the definition for these operators is a natural extension of $\partial_r:M^{1,p}_{r,\gamma-1}(\R^2) \rightarrow L^p_{r,\gamma}(\R^2)$, and $\partial_r+ \frac{1}{r}:M^{1,p}_{r,\gamma-1}(\R^2)\rightarrow L^p_{r,\gamma}(\R^2)$ since by duality
\[ \partial_r u (v) = \llangle u, (\partial_r)^* v \rrangle = \llangle u, (\partial_r + 1/r) v \rrangle, \qquad \forall u \in M^{1,p}_{r,\gamma-1}(\R^2), \forall v \in M^{1,q}_{r,1-\gamma}(\R^2)\]
\begin{Lemma}
Let $p \in (1,\infty)$, then the operator $\partial_r :L^p_{r,\gamma-1}(\R^2) \rightarrow M^{-1,p}_{r,\gamma}(\R^2)$ is 
\begin{itemize}
\item injective for $\gamma>1-2/p$.
\item Fredholm with $\ker = \{1 \}$ and index $i =1$ for $\gamma< 1-2/p$.
\end{itemize}
\end{Lemma}
\begin{proof}
To show it is injective for $\gamma>1-2/p$, suppose there is a $u \in L^p_{r,\gamma-1}(\R^2)$ such that $\partial_r u =0$. Then using a sequence $\{ u_n \} \in C^\infty_0(\R^2)$ such that $u_n \rightarrow u$ in $L^p_{r,\gamma-1}(\R^2)$ we find that for all $v \in M^{1,q}_{r,-\gamma}(\R^2)$
\[ 0= \partial_r u = \llangle u, (\partial_r)^* v \rrangle = \lim_{n \rightarrow \infty} \llangle \partial_r u_n, v \rrangle. \]
It follows that $\partial_r u_n \rightarrow 0$ in $L^p_{r, \gamma}$. Since $\gamma>1-2/p$ we must have $u =0$.

The result for $\gamma < 1-2/p$ follows from the definition $\partial_r u (v) = \llangle u, (\partial_r +1/r) v \rrangle$, and Lemma \ref{l:dr1overr}.
\end{proof}
Similar arguments as in the above lemma show that
\begin{Lemma}
Given $p \in (1,\infty)$,  the operator $\partial_r+ \frac{1}{r} :L^p_{r,\gamma-1}(\R^2) \rightarrow M^{-1,p}_{r,\gamma}(\R^2)$ is 
\begin{itemize}
\item injective for $\gamma< 2-2/p$.
\item Fredholm with $\coker = \{1 \}$ and index $i =-1$ for $\gamma> 2-2/p$.
\end{itemize}
\end{Lemma}

\begin{Lemma}\label{l:drnotclosed}
Given $\gamma = 1-2/p$, the operator $\partial_r :M^{1,p}_{r,\gamma-1}(\R^2) \rightarrow L^p_{r,\gamma}(\R^2)$ does not have a closed range.
\end{Lemma}
\begin{proof}
Let $\psi(r) \in C^\infty_0(\R^2)$ be a radial function with $\supp \psi(r) \in B_1$ and $\psi(r) = 1 $ for $r<1/2$. Define $u_n (r) = \psi(r/n)/ \| \psi(r/n) \|_{L^p_{r,\gamma-1}}$ for $n \in \N$ and notice that $\partial_r u_n \rightarrow 0$ in $L^p_{r,\gamma}(\R^2)$ yet for $\gamma = 1-2/p$ the sequence $\{u_n\}$ does not converge in $L^p_{r,\gamma-1}(\R^2)$. It follows then that $\partial_r :M^{1,p}_{r,\gamma-1}(\R^2) \rightarrow L^p_{r,\gamma}(\R^2)$ does not have closed range.
\end{proof}

\begin{Lemma}\label{l:dr1rnotclosed}
Given $\gamma = 1-2/p$, the operator $\partial_r + \frac{1}{r} :M^{1,p}_{r,\gamma-1}(\R^2) \rightarrow L^p_{r,\gamma}(\R^2)$ does not have a closed range.
\end{Lemma}
\begin{proof}
The proof is similar to that of Lemma \ref{l:drnotclosed} only we use the following sequence instead: Let $\psi(r) \in C^\infty_0(\R^2)$ be a radial function with $\supp \psi(r) \in B_2$ and $\psi(r) = 1/r $ for $1< r<3/2$, and $\psi(r) = 0 $ for $r<1/2$. Then define $u_n (r) = \psi(r/n)/ \| \psi(r/n) \|_{L^p_{r,\gamma-1}(\R^2)}$.
\end{proof}

\bibliographystyle{amsalpha}	
\bibliography{Pace2d}

\end{document}